\documentclass[11pt,english]{amsart}
\usepackage[T1]{fontenc}
\usepackage[latin9]{inputenc}
\usepackage{geometry}
\usepackage{amstext}
\usepackage{amsthm}
\usepackage{amssymb}
\usepackage{graphicx}
\usepackage{setspace}
\makeatletter
\numberwithin{equation}{section}
\numberwithin{figure}{section}
\theoremstyle{plain}
\newtheorem{thm}{\protect\theoremname}[section]
\theoremstyle{plain}
\newtheorem{cor}[thm]{\protect\corollaryname}
\theoremstyle{plain}
\newtheorem{prop}[thm]{\protect\propositionname}
\theoremstyle{remark}
\newtheorem{rem}[thm]{\protect\remarkname}
\theoremstyle{plain}
\newtheorem{fact}[thm]{\protect\factname}
\theoremstyle{plain}
\newtheorem{lem}[thm]{\protect\lemmaname}

\date{}
\usepackage{mathrsfs}
\pdfpageattr{/Group << /S /Transparency /I true /CS /DeviceRGB>>}
\usepackage{ae,aecompl} 
\usepackage{enumerate}
\usepackage{bbm}
\usepackage{hyperref}

\let\originalleft\left
\let\originalright\right
\renewcommand{\left}{\mathopen{}\mathclose\bgroup\originalleft}
\renewcommand{\right}{\aftergroup\egroup\originalright}

\hypersetup{
                    colorlinks=true,
                    linkcolor=red,
                    citecolor=blue,
                    linktocpage=true,
                    pdfstartview=FitH,
                }

\makeatother

\usepackage{babel}
\providecommand{\corollaryname}{Corollary}
\providecommand{\factname}{Fact}
\providecommand{\lemmaname}{Lemma}
\providecommand{\propositionname}{Proposition}
\providecommand{\remarkname}{Remark}
\providecommand{\theoremname}{Theorem}

\begin{document}
\global\long\def\d#1{\,{\rm d}#1}

\global\long\def\R{\mathbb{R}}

\global\long\def\C{\mathbb{C}}

\global\long\def\Z{\mathbb{Z}}

\global\long\def\N{\mathbb{N}}

\global\long\def\Q{\mathbb{Q}}

\global\long\def\T{\mathbb{T}}

\global\long\def\F{\mathbb{F}}

\global\long\def\vp{\varphi}

\global\long\def\Sph{\mathbb{S}}

\global\long\def\sub{\subseteq}

\global\long\def\one{\mathbbm1}

\global\long\def\vol#1{\text{vol}\left(#1\right)}

\global\long\def\EE{\mathbb{E}}

\global\long\def\sp{{\rm sp}}

\global\long\def\iprod#1#2{\langle#1,\,#2\rangle}

\global\long\def\uball{B_{2}^{n}}

\global\long\def\conv#1{{\rm conv}\left(#1\right)}

\global\long\def\met#1{{\rm M}\left(#1\right)}

\global\long\def\supp#1{{\rm supp}\left(#1\right)}

\global\long\def\eps{\varepsilon}

\global\long\def\PP{\mathbb{P}}

\global\long\def\vein#1{{\rm vein}\left(#1\right)}

\global\long\def\fvein#1{{\rm vein^{*}}\left(#1\right)}

\global\long\def\bfvein#1{{\rm \mathbf{vein^{*}}}\left(\mathbf{#1}\right)}

\global\long\def\ill#1{{\rm ill\left(#1\right)}}

\global\long\def\fill#1{{\rm ill^{*}\left(#1\right)}}

\global\long\def\ovr#1{{\rm ovr\left(#1\right)}}

\global\long\def\norm#1{\left\Vert #1\right\Vert }

\global\long\def\proj{{\rm Pr}}

\global\long\def\d#1{\,{\rm d}#1}

\global\long\def\R{\mathbb{R}}

\global\long\def\C{\mathbb{C}}

\global\long\def\Z{\mathbb{Z}}

\global\long\def\N{\mathbb{N}}

\global\long\def\Q{\mathbb{Q}}

\global\long\def\T{\mathbb{T}}

\global\long\def\F{\mathbb{F}}

\global\long\def\vp{\varphi}

\global\long\def\Sph{\mathbb{S}}

\global\long\def\sub{\subseteq}

\global\long\def\one{\mathbbm1}

\global\long\def\vol#1{\text{vol}\left(#1\right)}

\global\long\def\EE{\mathbb{E}}

\global\long\def\sp{{\rm sp}}

\global\long\def\iprod#1#2{\langle#1,\,#2\rangle}

\global\long\def\uball{B_{2}^{n}}

\global\long\def\conv#1{{\rm conv}\left(#1\right)}

\global\long\def\met#1{{\rm M}\left(#1\right)}

\global\long\def\mfloat#1#2{{\rm M}_{#1}\left(#2\right)}

\global\long\def\mwfloat#1#2#3{{\rm M}_{#1}\left(#2,#3\right)}

\global\long\def\supp#1{{\rm supp}\left(#1\right)}

\global\long\def\eps{\varepsilon}

\global\long\def\PP{\mathbb{P}}

\global\long\def\vein#1{{\rm vein}\left(#1\right)}

\global\long\def\fvein#1{{\rm vein^{*}}\left(#1\right)}

\global\long\def\bfvein#1{{\rm \mathbf{vein^{*}}}\left(\mathbf{#1}\right)}

\global\long\def\ill#1{{\rm ill\left(#1\right)}}

\global\long\def\fill#1{{\rm ill^{*}\left(#1\right)}}

\global\long\def\ovr#1{{\rm ovr\left(#1\right)}}

\global\long\def\norm#1{\left\Vert #1\right\Vert }

\global\long\def\proj{{\rm Pr}}

\global\long\def\ra{\Rightarrow}

\global\long\def\H#1#2{H^{+}\left(#1,\,#2\right)}

\title{Ulam floating bodies}

\author{Han Huang}

\address{Department of Mathematics, University of Michigan, Ann Arbor, Michigan }

\email{sthhan@umich.edu}

\author{Boaz A. Slomka}

\address{Department of Mathematics, University of Michigan, Ann Arbor, Michigan }

\email{bslomka@umich.edu}

\author{Elisabeth M. Werner}

\address{Department of Mathematics, Case Western Reserve University, Cleveland,
Ohio }

\email{\noindent {\small{}elisabeth.werner@case.edu}}

\thanks{The third author is partially supported by NSF grant DMS-1504701}
\begin{abstract}
We study a new construction of bodies from a given convex body in
$\R^{n}$ which are isomorphic to (weighted) floating bodies. We establish
several properties of this new construction, including its relation
to $p$-affine surface areas. We show that these bodies are related
to Ulam\textquoteright s long-standing floating body problem which
asks whether Euclidean balls are the only bodies that can float, without
turning, in any orientation.
\end{abstract}


\maketitle

\section{Introduction}

\subsection{Metronoids}

Let $K$ be a convex body in $\R^{n}$ (i.e. a compact convex set
with non-empty interior), and denote its Lebesgue volume by $\left|K\right|$.
The purpose of this paper is to study a new family of convex bodies
$\mfloat{\delta}K$ associated to $K$, where $0<\delta<\left|K\right|$
is a parameter.

The construction of this family arises from the notion of metronoids
which was recently introduced in \cite{Huang2017} in order to study
extensions of problems concerning the approximation of convex bodies
by polytopes. Given a Borel measure $\mu$ on $\R^{n}$, the {\em
metronoid} associated to $\mu$ is the convex set defined by 
\[
\met{\mu}=\bigcup_{\substack{0\le f\le1,\\
\int_{\R^{n}}f\d{\mu}=1
}
}\left\{ \int_{\R^{n}}yf\left(y\right)\d{\mu}\left(y\right)\right\} ,
\]
where the union is taken over all functions $0\le f\le1$ for which
$\int_{\R^{n}}f\d{\mu=1}$ and $\int_{\R^{n}}yf\left(y\right)\d{\mu}\left(y\right)$
exists. Note that for a discrete measure of the form $\sum_{i=1}^{N}\delta_{x_{i}}$,
the corresponding metronoid is the convex hull of $x_{1},\dots,x_{N}$.
Hence $\met{\mu}$ can be thought of as a fractional extension of
the convex hull.

\subsection{Ulam's floating body}

Our main object $\mfloat{\delta}K$ is the metronoid generated by
the uniform measure on $K$ with total mass $\delta^{-1}\left|K\right|$.
Namely, let $\one_{K}$ be the characteristic function of $K$, and
$\mu$ the measure whose density with respect to Lebesgue measure
is $\delta^{-1}\one_{K}$. Then $\mfloat{\delta}K:=\met{\mu}$. It
turns out that $M_{\delta}\left(K\right)$ is intimately related to
the following long-standing problem proposed by Ulam, see e.g., \cite{Auerbach1938,ScotBook81,CFG91,Gardner:2006}:
Is a solid of uniform density which floats in water in every position
a Euclidean ball? While counterexamples were found in $\R^{2}$ (convex
and non-convex) and $\R^{3}$ (only non-convex), this problem remains
open  in arbitrary dimensions. For a full account of the progress
made on this problem, see \cite{Varkonyi13} and references therein. 

As we show in Section \ref{sec:Ulam_problem} below, along with a
precise description of Ulam's problem, one can restate Ulam's problem
in terms of $M_{\delta}\left(K\right)$ as follows: If $M_{\delta}\left(K\right)$
is a Euclidean ball, must $K$ be a Euclidean ball as well? For that
reason, we call $M_{\delta}\left(K\right)$ an {\em Ulam floating
body}. As far as we know, this construction and its relation to Ulam's
problem is not mentioned anywhere in the literature. 

We also define weighted variations of $\mfloat{\delta}K$ where the
weight is given by a positive continuous function $\phi:K\to\R$.
Namely, we define 
\[
\mwfloat{\delta}K{\phi}:=\met{\frac{\phi\left(x\right)}{\delta}\one_{K}\left(x\right)\d x}.
\]

To understand $\mfloat{\delta}K$ geometrically, recall that a convex
body $K\sub\R^{n}$ is determined by its support function $h_{K}\left(\theta\right)=\max_{x\in K}\iprod x{\theta}$,
where $\iprod{\cdot}{\cdot}$ is the standard scalar product on $\R^{n}$.
For every direction $\theta\in\Sph^{n-1}$, let $H\left(\delta,\theta\right)$
be the hyperplane orthogonal to $\theta$ that cuts a set of volume
$\delta$ from $K$. That is 
\[
C_{\delta}\left(\theta\right)=K\cap\left\{ x\,:\,\iprod x{\theta}\ge\iprod{y_{\theta}}{\theta}\right\} 
\]
has volume $\delta$ for any $y_{\theta}\in H\left(\delta,\theta\right)$.
Then the barycenter of $C_{\delta}\left(\theta\right)$ is a point
on the boundary of $\mfloat{\delta}K$. More precisely, by \cite[Proposition 2.1]{Huang2017},
we have that for any direction $\theta$, 
\[
h_{\mfloat{\delta}K}\left(\theta\right)=\frac{1}{\delta}\int_{C_{\delta}\left(\theta\right)}\iprod x{\theta}\d x.
\]

As illustrated in Figure \ref{fig:Floating_Metro}, the body $\mfloat{\delta}K$
is closely related to the convex floating body $K_{\delta}$, introduced
independently in \cite{Barany_Larman:1988} and \cite{Schuett:1990}.
Using the above notation, we have that 
\[
K_{\delta}=\bigcap_{\theta\in\Sph^{n-1}}\left\{ x\,:\,\iprod x{\theta}\le\iprod{y_{\theta}}{\theta}\right\} ,
\]
which is a non-empty convex set for a sufficiently small $0<\delta$.
In fact, $\mfloat{\delta}K$ is isomorphic to $K_{\delta}$ in the
sense that $K_{\frac{e-1}{e}\delta}\sub\mfloat{\delta}K\sub K_{\frac{1}{e}}$.
We discuss this property in the more general case of weighted Ulam
floating bodies in Section \ref{sec:F_M} below (also see Theorem
\ref{thm:floating_metro_iso}).\bigskip{}

\begin{figure}[h]
\begin{centering}
\includegraphics[scale=1.1]{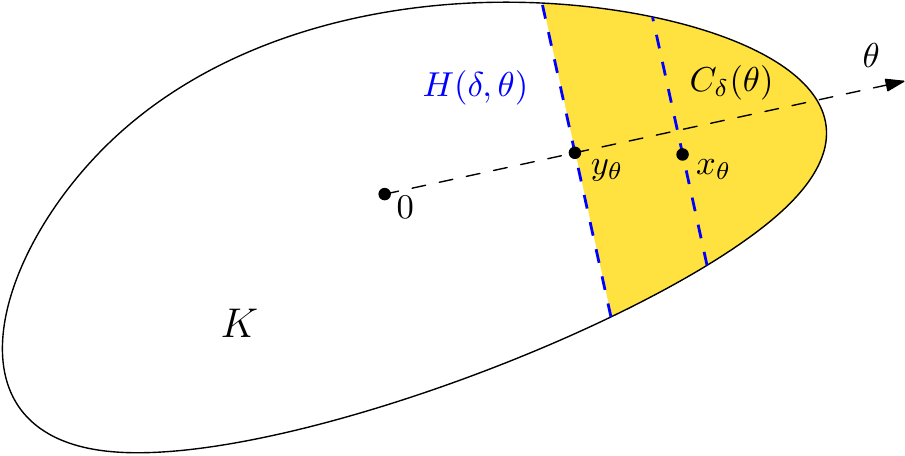} 
\par\end{centering}
\caption{\label{fig:Floating_Metro}$H\left(\delta,\theta\right)$ is the hyperplane
orthogonal to $\theta$ that cuts a set $C_{\delta}\left(\theta\right)$
of volume $\delta$ from a convex body $K$: $\left|C_{\delta}\left(\theta\right)\right|=\left|K\cap\left\{ x\,:\,\protect\iprod x{\theta}\ge\protect\iprod{y_{\theta}}{\theta}\right\} \right|=\delta$.
The point $x_{\theta}$ is the barycenter of $C_{\delta}\left(\theta\right)$.
Then 
\[
K_{\delta}\protect\sub K\cap\left\{ x\,:\,\protect\iprod x{\theta}\le\protect\iprod{y_{\theta}}{\theta}\right\} 
\]
while 
\[
\protect\mfloat{\delta}K\protect\sub K\cap\left\{ x\,:\,\protect\iprod x{\theta}\le\protect\iprod{x_{\theta}}{\theta}\right\} .
\]
}
\end{figure}

The convex floating body is a natural variation of Dupin's floating
body \cite{Dupin:1822} from 1822. Dupin's floating body $K_{\left[\delta\right]}$
is defined as the body whose boundary is the set of points that are
the barycenters of all the sections of $K$ of the form $K\cap H\left(\delta,\theta\right)$,
where $H\left(\delta,\theta\right)$ are the aforementioned hyperplanes
that cut a set of volume $\delta$ from $K$. However, while $K_{\delta}$
coincides with $K_{\left[\delta\right]}$ whenever $K_{\left[\delta\right]}$
is convex (e.g., for centrally-symmetric $K$, see \cite{Meyer-Reisner1991}),
in the non-centrally symmetric case, Dupin's floating body need not
be convex, as in the case of some triangles in $\R^{2}$ (see e.g.,
\cite{Leichtweiss:1998}). Restating the above, every point on the
boundary of $K_{\delta}$ is the barycenter of $K\cap H\left(\delta,\theta\right)$
for some $\theta$, but the converse holds only if Dupin's floating
body is convex.\medskip{}

Note that our construction $M_{\delta}\left(K\right)$ corresponds
nicely to both definitions, that of the floating body and that of
the convex floating body in the sense that it enjoys being convex
as well as having the property that a point is on the boundary of
$\mfloat{\delta}K$ if and only if it is the barycenter of a set of
volume $\delta$ that is cut off by a hyperplane. 

\subsection{Main results}

We present three main theorems concerning Ulam's floating bodies.
While the first result establishes an explicit relation between (weighted)
floating bodies and (weighted) Ulam's floating bodies, the other two
results are the analogous counterparts to the classical floating bodies.

\subsubsection{\textbf{Relation to floating bodies}}

Our first theorem shows that (weighted) Ulam's floating bodies are
isomorphic, in a sense, to (weighted) floating bodies. Weighted floating
bodies were introduced in \cite{Werner2002} (also see \cite{Besau2016,Besau_Werner:2018}
for recent applications) as follows. Let $K\sub\R^{n}$ be a convex
body, $0<\delta$, and $\phi:K\to\R$ be integrable and such that
$\phi>0$ almost everywhere with respect to Lebesgue measure. For
a hyperplane $H$ in $\R^{n}$, let $H^{\pm}$ be the half-spaces
separated by $H$. Then the weighted floating body $F_{\delta}\left(K,\phi\right)$
is defined as 
\begin{align*}
F_{\delta}\left(K,\phi\right)=\bigcap\left\{ H^{-}:\int_{H^{+}\cap K}\phi\left(x\right)\d x\leq\delta\right\} .
\end{align*}
Note that for $\phi\equiv1$, we have that $F_{\delta}\left(K,\phi\right)=K_{\delta}.$ 

We prove the following. 
\begin{thm}
\label{thm:floating_metro_iso}Let $K$ be a convex body in $\R^{n}$,
and let $\phi:K\to\R^{+}$ be an integrable log-concave function.
Then for all $0<\delta<\left|K\right|$, we have 
\[
F_{\frac{e-1}{e}\delta}\left(K,\phi\right)\sub\mwfloat{\delta}K{\phi}\sub F_{\frac{\delta}{e}}\left(K,\phi\right).
\]
In particular, for $\phi\equiv1$ we have that 
\[
K_{\frac{e-1}{e}\delta}\sub\mwfloat{\delta}K{\phi}\sub K_{\frac{\delta}{e}}.
\]
\end{thm}

\noindent We remark that for $\phi\equiv1$, Theorem \ref{thm:floating_metro_iso}
was proven in \cite{Huang2017}.

\subsubsection{\textbf{Smoothness of Ulam's floating bodies}}

Our second main result states that the boundary $\partial\mfloat{\delta}K$
of an Ulam floating body $\mfloat{\delta}K$ is always smoother than
the boundary of $K$.
\begin{thm}
\label{thm:Smoothness } Let $K\sub\R^{n}$ be a convex body, Suppose
that $\partial K\in C^{k}$ for some $k\ge0$. Then for any $0<\delta<\left|K\right|$,
we have that $\partial\mfloat{\delta}K\in C^{k+1}$.
\end{thm}

We remark that in the case of the convex floating body, an analogous
result to Theorem \ref{thm:Smoothness } is known only in the centrally-symmetric
case \cite{Meyer-Reisner1991}. The main reason for this is 
that the proof in \cite{Meyer-Reisner1991} relies on the above mentioned
fact that in the centrally-symmetric case the convex floating convex
body and Dupin's floating body coincide.

\subsubsection{\textbf{Affine Surface Area}}

The affine surface area was introduced by W.\ Blaschke \cite{Blaschke:1923}
in 1923 for smooth convex bodies in Euclidean space of dimensions
2 and 3, and extended to $\R^{n}$ by K.\ Leichtweiss \cite{Leichtweiss:1986}.
Given a convex body $K\sub\R^{n}$ with a sufficiently smooth boundary,
let $\kappa_{K}\left(x\right)$ be the Gaussian curvature at $x\in\partial K$,
and $\mu_{K}$ the surface area measure on $\partial K$. The affine
surface area of $K$ is defined by 
\[
as\left(K\right)=\int_{\partial K}\kappa_{K}\left(x\right)^{\frac{1}{n+1}}\d{\mu_{K}}.
\]
Even though it proved to be much more difficult to extend the notion
of affine surface area to general convex bodies than other notions,
like surface area measures or curvature measures, successively such
extensions were achieved, by e.g., K.\ Leichtweiss \cite{Leichtweiss:1986},
E.\ Lutwak \cite{Lutwak:1991}, who also proved the long conjectured
upper semicontinuity of affine surface area \cite{Lutwak:1991} and
by C.\ Schütt and E.\ Werner \cite{Schuett:1990} who showed that
the affine surface area arises as a limit of the volume difference
of the convex body and its floating body. All these extensions coincide
as was shown in \cite{Schuett:1993,Leichtweiss89}.

Affine surface area is among the most powerful tools in equiaffine
differential geometry (see B.\
Andrews \cite{Andrews:1996,Andrews:1999}, A.\ Stancu \cite{Stancu:2002,Stancu:2003},
M.\ Ivaki \cite{Ivaki:2014}, M.\ Ivaki and A.\ Stancu \cite{Ivaki:2013a}
and M. Ludwig and M. Reitzner \cite{Ludwig:2010}). It appears naturally
as the Riemannian volume of a smooth convex hypersurface with respect
to the affine metric (or Berwald-Blaschke metric), see e.g., the thorough
monograph of K.\ Leichtweiss \cite{Leichtweiss:1998} or the book
by K.\ Nomizu and T.\ Sasaki \cite{Nomizu:1994}. In particular
the upper semicontinuity proved to be critical in the solution of
the affine Plateau problem by N.\ S.\ Trudinger and X. J. Wang \cite{Trudinger:2005}.

Applications of affine surface areas have been manifold. For instance,
affine surface area appears in best and random approximation of convex
bodies by polytopes, see K.\ Böröczky Jr.\ \cite{Boeroeczky:2000,Boeroeczky:2000a},
P.\ M.\ Gruber \cite{Gruber:1988,Gruber:1993,Gruber:2007}, M.\ Ludwig
\cite{Ludwig:1999a}, M.\ Reitzner \cite{Reitzner:2005}, C.\
Schütt \cite{Schuett:1991,Schutt:1994} and J. Grote and E. Werner,
\cite{Grote_Werner:2018} and C.\ Schütt and E.\ Werner \cite{Schuett:2003}.
Furthermore, recent contributions indicate astonishing developments
which open up new connections of affine surface area to, e.g., concentration
of volume (e.g.\ \cite{Fleury:2007,Lutwak:2002a}), spherical and
hyperbolic spaces \cite{Besau_Werner:2016,Besau_Werner:2018}, geometric
inequalities \cite{Lutwak:1997,Werner:2008} and information theory
(e.g.\ \cite{Artstein-Avidan:2012,Caglar:2014,Lutwak:2004b,Lutwak:2005a,Werner:2012,Paouris_Werner:2012}).

The $\mathrm{L}_{p}$-affine surface area is a generalization of the
classical affine surface area and a central part in the $\mathrm{L}_{p}$-Brunn-Minkowski
theory. It was introduced by \ E.\ Lutwak \cite{Lutwak:1996} for
$p>1$ (see also D.\ Hug \cite{Hug:1996} and M.\ Meyer and E.\ Werner
\cite{Meyer:2000}) and extended for all $p\in\left[-\infty,\infty\right]$
in \cite{Schuett:2004}. For $-\infty<p<\infty$ , the $\mathrm{L}_{p}$-affine
surface area of a convex body $K\sub\R^{n}$ is given by 
\begin{equation}
as_{p}(K)=\int_{\partial K}\frac{\kappa_{K}(x)^{\frac{p}{n+p}}}{\langle x,N_{K}(x)\rangle^{\frac{n(p-1)}{n+p}}}d\mu_{K}(x),\label{pasa1-1}
\end{equation}
where $N_{K}\left(x\right)$ is the outer normal of $K$ at $x$.
For $p=\pm\infty$, it is given by 
\begin{equation}
as_{\pm\infty}(K)=\int_{\partial K}\frac{\kappa_{K}(x)}{\langle x,N_{K}(x)\rangle^{n}}d\mu_{K}(x).\label{pasa2-1}
\end{equation}
As in the case of the classical affine surface area, several geometric
extensions for the $L_{p}$-affine surface area have been proven.
We refer to \cite{Schuett:2004,Werner:2007} and references therein.
These extensions all involve a construction of a special family of
convex bodies $\left\{ K_{t}\right\} _{t>0}$ which is related to
a given convex body $K$, where the $L_{p}$-affine surface area can
be written as a limit involving their volume difference. 

We prove the following theorem which shows that this can also be achieved
using weighted Ulam floating bodies. 
\begin{thm}
\label{thm:main} Let $K\sub\R^{n}$ be a convex body and $\phi:K\rightarrow\left(0,\infty\right)$
be a continuous function. Then 
\begin{equation}
\lim_{\delta\searrow0}\frac{\left|K\right|-\left|\text{\ensuremath{\mwfloat{\delta}K{\phi}}}\right|}{\delta^{\frac{2}{n+1}}}=c_{n}\int_{\partial K}\kappa_{K}\left(x\right)^{\frac{1}{n+1}}\phi\left(x\right)^{-\frac{2}{n+1}}\d{\mu_{K}\left(x\right)},\label{eq:main}
\end{equation}
where $c_{n}=2\frac{n+1}{n+3}\left(\frac{\left|B_{2}^{n-1}\right|}{n+1}\right)^{\frac{2}{n+1}}$,
and $B_{2}^{n}$ is the Euclidean unit ball in $\R^{n}$.
\end{thm}

\noindent For $-\infty\leq p\leq\infty$, $p\neq-n$, define the function
$\phi_{p}:\partial K\rightarrow[0,\infty]$ by 
\begin{equation}
\phi_{p}\left(x\right)=\frac{\iprod x{N_{K}\left(x\right)}^{\frac{n\left(n+1\right)\left(p-1\right)}{2\left(n+p\right)}}}{\kappa_{K}\left(x\right)^{\frac{n\left(p-1\right)}{2\left(n+p\right)}}}.\label{phi-p}
\end{equation}
Note that $\phi_{1}\left(x\right)=1$ for all $x\in\partial K$. If
$\kappa_{K}(x)=0$, which is the case, e.g., when $K=P$ is a polytope
and $x$ belongs to an $(n-1)$-dimensional facet of $P$, then 
\[
\phi_{p}\left(x\right)=\left\{ \begin{array}{cc}
\infty & \hskip5mmp>1\hskip1mm\mbox{or}\hskip1mmp<-n\\
0 & -n<p<1.
\end{array}\right.
\]
If $\kappa_{K}(x)=\infty$, which is the case, e.g., when $K=P$ is
a polytope and $x$ is a vertex of $P$, then 
\[
\phi_{p}\left(x\right)=\left\{ \begin{array}{cc}
0 & \hskip5mmp>1\hskip1mm\mbox{or}\hskip1mmp<-n\\
\infty & -n<p<1.
\end{array}\right.
\]
If $K$ and $p$ are such that $\phi_{p}$ is continuous on $\partial K$,
we extend $\phi_{p}$ to a continuous function on $K$ which we call
again $\phi_{p}$.

Applying Theorem \ref{thm:main} with $\phi_{p}$ yields the following
extension of $L_{p}$-affine surface areas. 
\begin{cor}
\label{cor:pASA}Let $K\sub\R^{n}$ be a convex body. If $\phi_{p}$
is continuous on $K$, then 
\[
\lim_{\delta\searrow0}\frac{\left|K\right|-\left|\text{\ensuremath{\mwfloat{\delta}K{\phi_{p}}}}\right|}{\delta^{\frac{2}{n+1}}}=c_{n}\ as_{p}(K).
\]
In particular, for $p=1$ we have 
\[
\lim_{\delta\searrow0}\frac{\left|K\right|-\left|\text{M}_{\delta}(K)\right|}{\delta^{\frac{2}{n+1}}}=c_{n}\ as_{1}(K).
\]
\end{cor}

\subsection{Some additional notation}

Throughout the paper we denote by $B_{2}^{n}\left(u,\rho\right)$
the Euclidean ball with radius $\rho>0$ centered at $u$. Let $\norm{\cdot}$
denote the standard Euclidean norm on $\R^{n}$. For $u,v\in\R^{n}$,
$\left[u,v\right]$ will denote the line segment between $u$ and
$v$. We denote the interior of a set $C\sub\R^{n}$ by ${\rm int}\left(C\right)$.
In the sequel, we will always assume that our convex body $K$ contains
the origin in its interior. Finally, $c,c_{0},c_{1},{\rm etc.}$ shall
denote absolute constants that may change from line to line. Let $O_{n}$
denote the orthogonal group of dimension $n$.\\

The paper is organized as follows. In Section \ref{sec:Properties}
we discuss some properties of Ulam's floating bodies, and prove Theorems
\ref{thm:floating_metro_iso} and \ref{thm:Smoothness }. Section
\ref{sec:ASA} is devoted for the proof of Theorem \ref{thm:main}.

\subsection*{Acknowledgements}

We thank Ning Zhang for pointing out Proposition \ref{prop:symmetry}
to us. We also thank Monika Ludwig for useful comments and references.
This material is based upon work supported by the National Science
Foundation under Grant No. DMS-1440140 while the authors were in residence
at the Mathematical Sciences Research Institute in Berkeley, California,
during the Fall 2017 semester.

\section{\label{sec:Properties}Properties of Ulam's floating bodies}

\subsection{Basic properties}

\noindent For $\theta\in\Sph^{n-1}$ and $d\in\R$, we define the
hyperplane orthogonal to $\theta$ at distance $d$ from the origin
by $H\left(\theta,\,d\right):=\left\{ x\in\R^{n}\,:\,\iprod x{\theta}=d\right\} $.
We also define one of the closed half-spaces determined by $H\left(\theta,\,d\right)$
by $H^{+}\left(\theta,\,d\right):=\left\{ x\in\R^{n}\,:\,\iprod x{\theta}\ge d\right\} $.
The function 
\begin{eqnarray*}
 &  & \Sph^{n-1}\times\mathbb{R}\longrightarrow\left[0,\int_{K}\phi(z)dz\right]\\
 &  & (\theta,d)\longrightarrow\text{\ensuremath{\delta\left(\theta,\,d\right)}}:=\int_{K\cap\H{\theta}d}\phi\left(z\right)\d z
\end{eqnarray*}
is continuous in the product metric. Observe also that the function
$\left(\theta,r\right)\rightarrow\left(\theta,\delta\left(\theta,\,r\right)\right)$
is a bijection from 
\begin{align*}
\left\{ \left(\theta,r\right)\,:\,\theta\in\Sph^{n-1},\,-h_{K}\left(-\theta\right)\le r\le h_{K}\left(\theta\right)\right\} 
\end{align*}
to $\Sph^{n-1}\times\left[0,\,\int_{K}\phi\left(x\right)\d x\right]$.
We denote 
\begin{eqnarray}
\left(\theta,\,\delta\right)\rightarrow\left(\theta,\,d\left(\theta,\delta\right)\right)\label{stetig}
\end{eqnarray}
as the inverse function of $\left(\theta,\,d\right)\rightarrow\left(\theta,\,\delta\left(\theta,\,d\right)\right)$,
which is also a continuous function. Abusing the notation we denote
\begin{equation}
\H{\theta}{\delta}:=\H{\theta}{d\left(\theta,\,\delta\right)},\label{hyper}
\end{equation}
Let $h_{\mwfloat{\delta}K{\phi}}\left(\theta\right)$ be the support
function of $\mwfloat{\delta}K{\phi}$. By definition of $\mwfloat{\delta}K{\phi}$,
\begin{equation}
h_{\mwfloat{\delta}K{\phi}}\left(\theta\right)=\max_{x\in\mwfloat{\delta}K{\phi}}\iprod{\theta}x=\sup_{0\le f\le1,\int_{K}\frac{f\left(y\right)\phi\left(y\right)}{\delta}\d y=1}\int_{K}\iprod y{\theta}\frac{f\left(y\right)}{\delta}\phi\left(y\right)dy.\label{eq:supportfunction}
\end{equation}
It follows from \cite[Proposition 2.1]{Huang2017} that the maximum
in the above equation is attained for the function 
\[
f={\bf \one}_{K\cap H^{+}\left(\theta,\,\delta\right)}
\]
and this maximal function is unique as $\phi\left(y\right){\bf \one}_{K}\d y$
is absolutely continuous with respect to Lebesgue measure. Thus we
have the following proposition which is essentially a restatement
of Proposition 2.1 of \cite{Huang2017}. 
\begin{prop}
\label{prop:ExtremePointOfMetronoid}Let $K\sub\R^{n}$ be a convex
body and $\phi:K\rightarrow\left(0,\infty\right)$ be a continuous
function. Let $\theta\in\Sph^{n-1}$ and $\delta\in\left(0,\int_{K}\phi\left(y\right)\d y\right)$.
Then, the barycenter of $K\cap\H{\theta}{\delta}$ with respect to
the weight function $\phi$, 
\[
x_{K,\,\phi}\left(\theta,\,\delta\right):=\frac{\int_{K\cap\H{\theta}{\delta}}y\phi\left(y\right)\d y}{\delta}
\]
is the unique point in $\partial\mwfloat{\delta}K{\phi}$ with normal
$\theta$. In particular, $\mwfloat{\delta}K{\phi}$ is strictly convex.
Moreover, 
\[
h_{\mwfloat{\delta}K{\phi}}\left(\theta\right)=\frac{\int_{K\cap\H{\theta}{\delta}}\iprod{\theta}y\phi\left(y\right)\d y}{\delta}.
\]
Extending by limit, $h_{\mwfloat{\delta}K{\phi}}$ is a continuous
function on $\Sph^{n-1}\times\left[0,\int_{K}\phi\left(y\right)\d y\right]$
and $h_{\mwfloat 0K{\phi}}$ is the support function $h_{K}$ of $K$. 
\end{prop}

We remark that we will use $x\left(\theta,\,\delta\right)$ in short
for $x_{K,\,\phi}\left(\theta,\,\delta\right)$ whenever there is
no ambiguity (which is actually everywhere, except for the proof of
Theorem \ref{thm:Smoothness }).
\begin{proof}
We only need to show that $h_{\mwfloat{\delta}K{\phi}}$ is continuous
as a function of $\theta$ and $\delta$. We put $g\left(\theta,\,d\right)=\int_{K\cap H^{+}\left(\theta,\,d\right)}\iprod{\theta}y\phi\left(y\right)\d y$.
Then $g$ is continuous in the product metric. By the above, the function
$\left(\theta,\,\delta\right)\rightarrow\left(\theta,\,d\left(\theta,\delta\right)\right)$
is continuous in the product metric. Now 
\[
h_{\mwfloat{\delta}K{\phi}}\left(\theta\right)=\frac{g\left(\theta,\,d\left(\theta,\,\delta\right)\right)}{\delta},
\]
and therefore it is continuous for $0<\delta\le\int_{K}\phi\left(y\right)\d y$,
$\theta\in\Sph^{n-1}$. Moreover, for all $\theta\in\Sph^{n-1}$ and
for all $\delta\in(0,\int_{K}\phi\left(y\right)\d y]$, 
\[
d\left(\theta,\delta\right)\le h_{\mwfloat{\delta}K{\phi}}\left(\theta\right)\le h_{K}\left(\theta\right).
\]
 Note that for $\delta=0$, $d\left(\theta,0\right)=h_{K}(\theta)$.
Let $\theta_{0}\in\Sph^{n-1}$ be fixed. For $\eps>0$, there exists
an open ball $O$ containing $\left(\theta_{0},\,0\right)\in\Sph^{n-1}\times\left[0,\int_{K}\phi\left(y\right)\d y\right]$
such that for $\left(\theta_{1},\,\delta_{1}\right)\in O$ we have
$\left|h_{K}\left(\theta_{0}\right)-d\left(\theta_{1},\,\delta_{1}\right)\right|<\eps$.
Thus, we conclude that $\left|h_{K}\left(\theta_{0}\right)-h_{\mwfloat{\delta_{1}}K{\phi}}\left(\theta_{1}\right)\right|<\eps$
and hence $h_{\mwfloat{\delta}K{\phi}}\left(\theta\right)$ is continuous
at $\left(\theta_{0},\,0\right)$ if we define $h_{\mwfloat 0K{\phi}}\left(\theta_{0}\right):=h_{K}\left(\theta_{0}\right)$.
\end{proof}

\subsection{\label{sec:Ulam_problem}Ulam's floating body problem }

Let $K\sub\R^{n}$ be a body with a uniform density $0<\rho<1$. Suppose
we put $K$ in a liquid of uniform density $1$, such that the surface
of the liquid is orthogonal to the direction $u$. Let $g$ be the
barycenter of $K$, and $b$ its center of buoyancy, that is the barycenter
of the portion of $K$ which is submerged in the liquid. We say that
$K$ floats in equilibrium in direction $u$ if the barycenter of
$K$ is directly above its buoyancy center, namely $g-b$ is parallel
to $u$. 

A well-known fact in hydrostatics which was pointed out to us by Ning
Zhang (see e.g., \cite[Theorem 2]{Gilbert91}) states that if a body
floats in liquid, then its barycenter, its center of buoyancy, and
the barycenter of the portion of the body that is above the surface
of the liquid, are all collinear. In terms of $M_{\delta}\left(K\right)$,
this property translates to the following proposition:
\begin{prop}
\label{prop:symmetry}Let $K\sub\R^{n}$ be a convex body with ${\rm bar}\left(K\right)=0$
and $\left|K\right|=1$. Then, $\mfloat{1-\delta}K=-\frac{\delta}{1-\delta}\mfloat{\delta}K$. 
\end{prop}

\begin{rem}
An immediate consequence of the above proposition is that for any
convex body $K\sub\R^{n}$, $\mfloat{\frac{1}{2}}K$ is centrally-symmetric.
Moreover, by Theorem \ref{thm:floating_metro_iso} and Proposition
\ref{prop:Relation_Zp}, it follows that $\mfloat{\frac{1}{2}}K$
is isomorphic to $B_{2}^{n}$.
\end{rem}

\begin{proof}
Recall that $h_{M_{\delta}\left(K\right)}\left(\theta\right)=\iprod{x\left(\theta,\delta\right)}{\theta}$
where
\[
x\left(\theta,\,\delta\right):=\frac{\int_{K\cap\H{\theta}{\delta}}y\d y}{\delta}
\]
and $H^{+}\left(\theta,\,\delta\right)$ is the half space in direction
$\theta$ such that $\left|K\cap H^{+}\left(\theta,\,\delta\right)\right|=\delta$.
Observe that 
\begin{align*}
0 & ={\rm bar}\left(K\right)=\int_{K}x\d x=\int_{K\cap H^{+}\left(\theta,\,\delta\right)}x\d x+\int_{K\cap H^{-}\left(\theta,\,\delta\right)}x\d x,
\end{align*}
which is equivalent to 

\[
0=\delta x\left(\text{\ensuremath{\theta},\,\ensuremath{\delta}}\right)+\left(1-\delta\right)x\left(-\theta,1-\delta\right).
\]

\noindent Therefore, $x\left(-\theta,1-\delta\right)=-\frac{\delta}{1-\delta}x\left(\theta,\,\delta\right)$,
which is equivalent to $\mfloat{1-\delta}K=-\frac{\delta}{1-\delta}\mfloat{\delta}K$.
\end{proof}
As mentioned in the introduction, Ulam's long-standing floating problem
asks whether the only body of uniform density that floats in equilibrium
in every orientation must be a Euclidean ball. A direct consequence
of Proposition \ref{prop:symmetry} is that Ulam's floating problem
can be restated in terms of $M_{\delta}\left(K\right)$:
\begin{cor}
Ulam's floating problem is equivalent to the following problem: Suppose
$M_{\delta}\left(K\right)$ is a Euclidean ball. Must $K$ be a Euclidean
ball?
\end{cor}

We remark that this new form of Ulam's problem remains open if one
replaces $M_{\delta}\left(K\right)$ with the convex floating body
$K_{\delta}$. Another related open problem asks whether a convex
body $K$ is centrally-symmetric if and only if $K_{\delta}$ is symmetric.
When replaced with $M_{\delta}\left(K\right)$, this problem seems
also interesting. Note that Auerbach's counterexample in \cite{Auerbach1938}
to Ulam's problem in the plane, provides an example for a non-centrally-symmetric
convex body $K\sub\R^{2}$ for which $M_{\delta}\left(K\right)$ is
a Euclidean ball, thus answer both of the above problems in this case. 

\subsection{\label{sec:F_M}Connection to floating bodies.}

We begin with the proof of Theorem \ref{thm:floating_metro_iso}:

\begin{proof}[Proof of Theorem \ref{thm:floating_metro_iso} ]
  By Proposition \ref{prop:ExtremePointOfMetronoid} we have that
\[
h_{\mwfloat{\delta}K{\phi}}\left(\theta\right)=\frac{1}{\delta}\int_{K\cap\left\{ y\in\R^{n}\,:\,\iprod y{\theta}\ge d\left(\theta,\,\delta\right)\right\} }\iprod x{\theta}\phi\left(x\right)\d x\ge d\left(\theta,\,\delta\right)\ge h_{F_{\delta}\left(K,\,\phi\right)}\left(\theta\right).
\]
Therefore, $F_{\delta}\left(K,\phi\right)\sub\mwfloat{\delta}K{\phi}$. 

Fix $\delta>0$ and $\theta\in\Sph^{n-1}$. For $\beta\in\Sph^{n-1}$,
let $H_{\beta}^{+}:=\left\{ y\in\R^{n}\,:\,\iprod y{\beta}\ge\iprod{x\left(\theta,\,\delta\right)}{\beta}\right\} .$
Consider the function $g_{\beta}\left(t\right):=\int_{\left\{ y\,:\,\iprod y{\beta}=t\right\} }{\bf 1}_{K\cap H^{+}\left(\theta,\,\delta\right)}\left(y\right)\phi\left(y\right)\d y$.
Since $\phi$ is log-concave, it follows by Prékopa-Leindler's inequality
that $g_{\beta}$ is also log-concave. By \cite[Lemma 5.4]{Lovasz:2007}
(a generalization of Grünbaum's inequality), we have that 
\[
\frac{1}{e}\int g_{\beta}\left(t\right)\d t\le\int_{t\ge\iprod{x\left(\theta,\,\delta\right)}{\beta}}g_{\beta}\left(t\right)\d t\le\left(1-\frac{1}{e}\right)\int g_{\beta}\left(t\right)\d t
\]
or equivalently, 
\[
\frac{1}{e}\int_{K\cap H^{+}\left(\theta,\,\delta\right)}\phi\left(y\right)\d y\le\int_{H_{\beta}^{+}\cap K\cap H^{+}\left(\theta,\,\delta\right)}\phi\left(y\right)\d y\le\left(1-\frac{1}{e}\right)\int_{K\cap H^{+}\left(\theta,\,\delta\right)}\phi\left(y\right)\d y.
\]
Taking $\beta=\theta$, we have $H_{\theta}^{+}\cap K\cap H^{+}\left(\theta,\,\delta\right)=H_{\theta}^{+}\cap K$.
Since $\int_{H_{\theta}^{+}\cap K}\phi\left(y\right)\d y\le\left(1-\frac{1}{e}\right)\delta$,
we obtain 
\begin{align*}
h_{F_{\left(1-\frac{1}{e}\right)\delta}\left(K,\,\phi\right)}\left(\theta\right)\le d\left(\theta,\,\left(1-\frac{1}{e}\right)\delta\right) & \le\iprod{x\left(\theta,\,\delta\right)}{\theta}=h_{\mwfloat{\delta}K{\phi}}\left(\theta\right),
\end{align*}
and thus $F_{\left(1-\frac{1}{e}\right)\delta}\left(K,\,\phi\right)\sub\mwfloat{\delta}K{\phi}.$
\begin{figure}[h]
\begin{centering}
\includegraphics[scale=0.8]{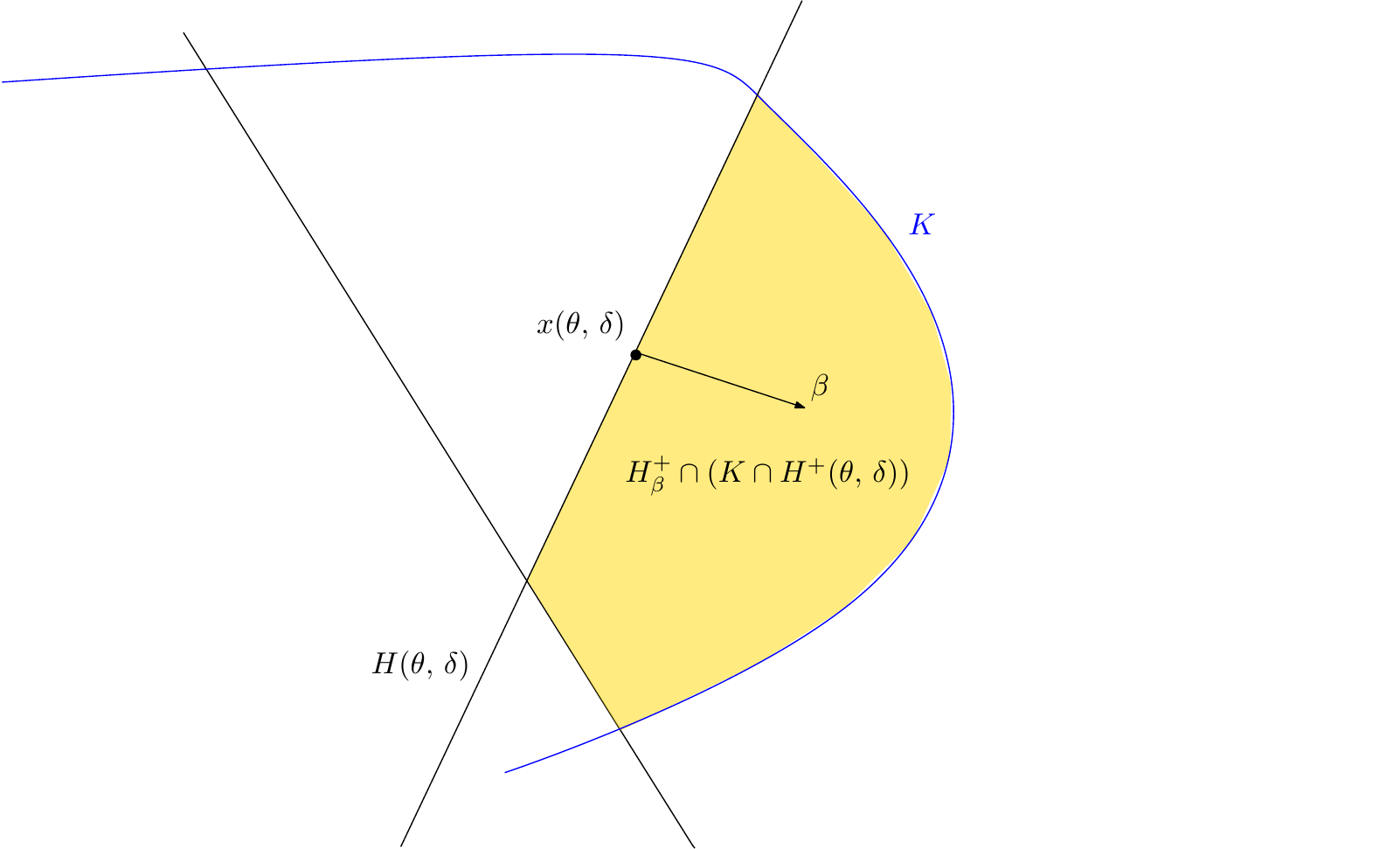}
\par\end{centering}
\caption{\label{fig:beta}}
\end{figure}
On the other hand (see Figure \ref{fig:beta}), for $\beta\in\Sph^{n-1}$
we have 
\[
\int_{H_{\beta}^{+}\cap K}\phi\left(y\right)\d y\ge\int_{H_{\beta}^{+}\cap K\cap H^{+}\left(\theta,\,\delta\right)}\phi\left(y\right)\d y\ge\frac{\delta}{e}=\int_{H^{+}\left(\beta,\,\frac{\delta}{e}\right)\cap K}\phi\left(y\right)\d y.
\]
Hence, $d\left(\beta,\,\frac{\delta}{e}\right)\ge\iprod{x\left(\theta,\,\delta\right)}{\beta}$.
Therefore we have 
\[
x\left(\theta,\,\delta\right)\in\bigcap_{\beta\in\Sph^{n-1}}\left\{ y\,:\,\iprod y{\beta}\le d\left(\theta,\,\frac{\delta}{e}\right)\right\} =F_{\frac{\delta}{e}}\left(K,\,\phi\right).
\]
Since $\mwfloat{\delta}K{\phi}$ and $F_{\frac{\delta}{e}}\left(K,\,\phi\right)$
are convex sets, we conclude that $\mwfloat{\delta}K{\phi}\sub F_{\frac{\delta}{e}}\left(K,\,\phi\right).$
\end{proof}
\noindent 

The $L_{p}$ centroid bodies were introduced by Lutwak and Zhang \cite{Lutwak:1997}
(using a different normalization) as follows: For a convex body $K$
in $\mathbb{R}^{n}$ of volume $1$ and $1\leq p\leq\infty$, the
$L_{p}$ centroid body $Z_{p}(K)$ is this convex body whose support
function is given by: 
\begin{equation}
h_{Z_{p}(K)}(\theta)=\left(\int_{K}|\langle x,\theta\rangle|^{p}dx\right)^{1/p}.\label{def:Zp}
\end{equation}

\noindent It is known that the floating body $K_{\delta}$ is close
to some {\em $L_{p}$ centroid body} of $K$. More precisely, one
has:
\begin{thm}
\label{thm: ZpFloating}( \cite[Theorem 2.2]{Paouris_Werner:2012})
Let $K$ be a symmetric convex body of volume $1$. For $\delta\in\left(0,\frac{1}{2}\right)$,
we have 
\[
c_{1}Z_{\log\left(\frac{e}{2\delta}\right)}\left(K\right)\sub K_{\delta}\sub c_{2}Z_{\log\left(\frac{e}{2\delta}\right)}\left(K\right),
\]
where $c_{1},c_{2}>0$ are universal constants. 
\end{thm}

\noindent We obtain a similar result for Ulam floating bodies:
\begin{prop}
\label{prop:Relation_Zp}Let $K$ be a symmetric convex body in $\R^{n}$
of volume $1$ . Then there is an absolute constant $c_{1}>0$ such
that for all $\delta<\frac{1}{e}$
\[
c_{1}Z_{\log\left(\frac{e}{2\delta}\right)}\left(K\right)\sub K_{\delta}\sub\mfloat{\delta}K\sub eZ_{\log\left(\frac{1}{\delta}\right)}\left(K\right).
\]
\end{prop}

\begin{proof}
The first inclusion holds by Theorem \ref{thm: ZpFloating}. The second
one, $K_{\delta}\sub\mfloat{\delta}K$, follows from Theorem \ref{thm:floating_metro_iso}.
By Hölder's inequality, we have for $p\in\left[1,\,\infty\right]$,
\begin{align*}
\int_{K\cap H^{+}\left(\theta,\,\delta\right)}\iprod y{\theta}\d y & \le\left(\int_{K\cap H^{+}\left(\theta,\,\delta\right)}1^{q}\d y\right)^{\frac{1}{q}}\left(\int_{K}\left|\iprod{\theta}y\right|^{p}\d y\right)^{\frac{1}{p}}\\
 & =\delta^{\frac{1}{q}}h_{Z_{p}\left(K\right)}\left(\theta\right),
\end{align*}
where $q$ satisfies $\frac{1}{p}+\frac{1}{q}=1$. Dividing both sides
by $\delta$, we get 
\[
h_{\mfloat{\delta}K}\left(\theta,\,\delta\right)\text{\ensuremath{\le}}\left(\frac{1}{\delta}\right)^{\frac{1}{p}}h_{Z_{p}\left(K\right)}\left(\theta\right).
\]
Putting $p=\log\left(\frac{1}{\delta}\right)$ yields 
\[
h_{\mfloat{\delta}K}\left(\theta,\,\delta\right)\le e\ h_{Z_{\log\left(\frac{1}{\delta}\right)}\left(K\right)}\left(\theta\right).
\]
Therefore, we have that
\[
\mfloat{\delta}K\sub e\ Z_{\log\left(\frac{1}{\delta}\right)}\left(K\right).
\]
\end{proof}

\subsection{Smoothness of Ulam floating bodies}

In this section we prove Theorem \ref{thm:Smoothness }. To this end,
let $\rho_{v}\left(\cdot\right)$ denote the radial function of $K$
with center $v$. That is, 
\[
\rho_{v}\left(\theta\right)=\max\left\{ r\in\R^{+}\,:\,v+r\theta\in K\right\} .
\]
We will need the following fact, which can be found implicitly in
e.g., \cite{Schneider:1993}.
\begin{fact}
\label{fact:smooth}Let $K\sub\R^{n}$ be a convex body. Then, the
following are equivalent:
\end{fact}

\begin{enumerate}
\item $K$ has $C^{k}$ boundary;
\item The function $\left(v,\,\theta\right)\rightarrow\rho_{v}\left(\theta\right)$
is $C^{k}$ for every $v\in{\rm int}\left(K\right)$ and $\theta\in\Sph^{n-1}$;
\item There exists $v\in{\rm int}\left(K\right)$ such that $\theta\rightarrow\rho_{v}\left(\theta\right)$
is $C^{k}$.
\end{enumerate}
\begin{proof}[Proof of Theorem \ref{thm:Smoothness }]
For $a\in\R^{n}\setminus\left\{ 0\right\} $, let $H:=\left\{ x\,:\,\iprod xa=1\right\} $,
$\delta\left(a\right)=\left|K\cap\left\{ \iprod xa\ge1\right\} \right|$,
and $U\left(a\right):=\int_{K\cap\left\{ \iprod xa\ge1\right\} }x\d x$.
We would like to show that 
\begin{align}
\nabla\delta\left(a\right) & =\frac{1}{\norm a}\int_{K\cap H}x\d x\label{eq:Ddelta}\\
DU & =\frac{1}{\norm a}\left(\int_{K\cap\left\{ \iprod xa=1\right\} }x_{i}x_{j}\d x\right)_{i,j\in\left[n\right]}.\label{eq:DU}
\end{align}
Equation \eqref{eq:Ddelta} was proved in \cite[Lemma 5]{Meyer1989}.
Using the same ideas, we prove \eqref{eq:DU} as follows. Pick a direction
$\theta$ so that $\theta$ is not parallel to $a$, and let $H_{\eps}:=\left\{ x\,:\,\iprod x{a+\eps\theta}=1\right\} .$
As illustrated in Figure \ref{fig:Kpm}, we also define: 
\begin{align*}
K_{-}\left(\eps\right) & ={\rm int}\left(K\right)\cap\left\{ y\in\R^{n}\,:\,\iprod ya\ge1,\,\iprod y{a+\eps\theta}\le1\right\} ,\\
K_{+}\left(\eps\right) & ={\rm int}\left(K\right)\cap\left\{ y\in\R^{n}\,:\,\iprod ya\le1,\,\iprod y{a+\eps\theta}\ge1\right\} .
\end{align*}
\begin{figure}[h]
\begin{centering}
\includegraphics[scale=0.8]{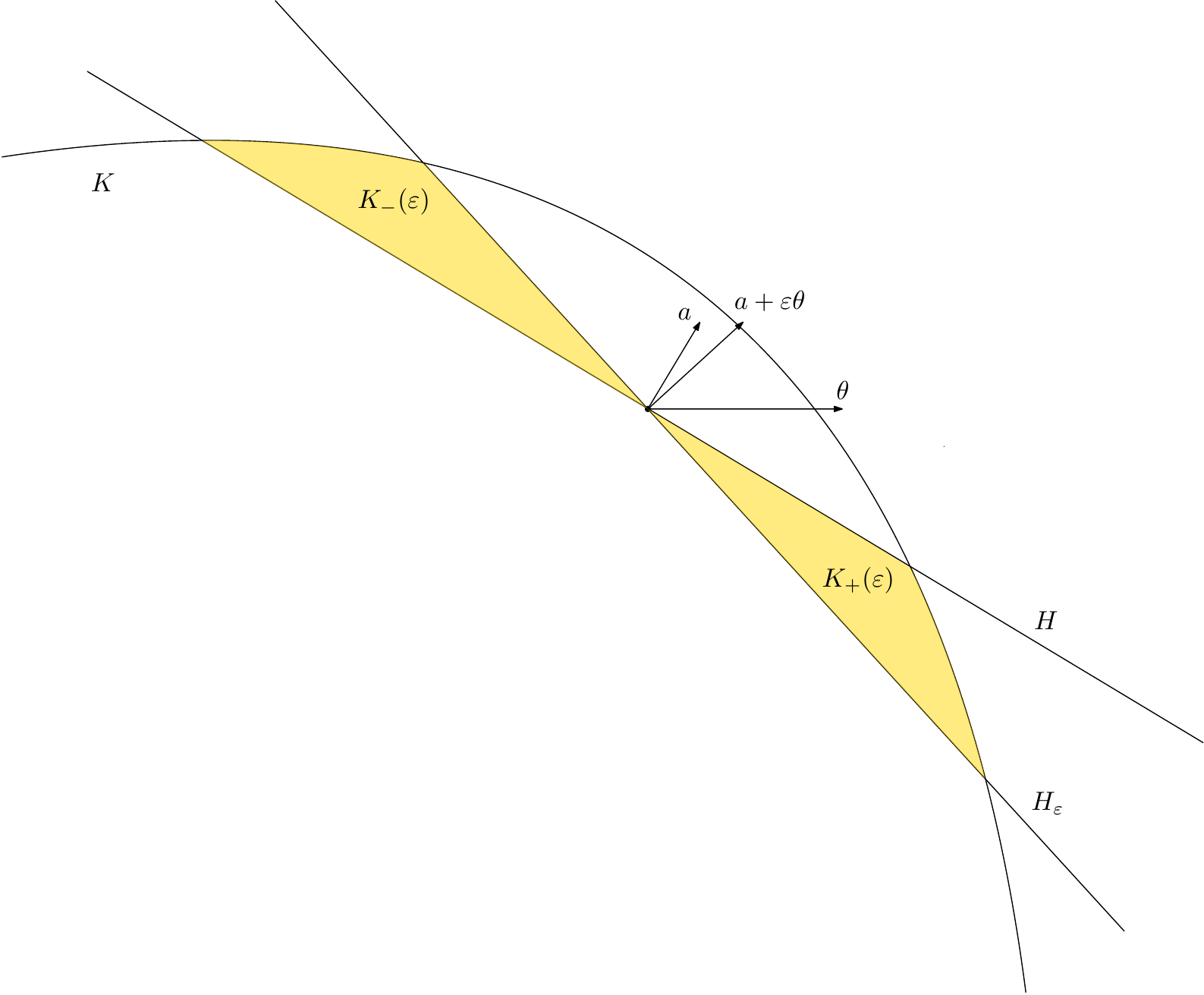}
\par\end{centering}
\centering{}\caption{\label{fig:Kpm}}
\end{figure}

Let $U_{j}$ denote the $j$th coordinate of $U$. We have 
\[
U_{j}\left(a+\eps\theta\right)-U_{j}\left(a\right)=\int_{K_{+}\left(\eps\right)}\iprod x{e_{j}}\d x-\int_{K_{-}\left(\eps\right)}\iprod x{e_{j}}\d x.
\]

From now on we choose $\eps>0$ small enough so that $\iprod a{a+\eps\theta}>0$.
For $y\in\R^{n}$, we write $y$ uniquely in the form $x+t\frac{a}{\norm a}$,
where $x=y+\frac{1-\iprod ya}{\iprod aa}a$ and $t=-\frac{1-\iprod ya}{\iprod aa}\norm a$.
Notice that $x\in H$. Then, 

\begin{align*}
 & \left\{ y\in\R^{n}\,:\,\iprod ya\ge1,\,\iprod y{a+\eps\theta}\le1\right\} =\\
 & \left\{ x+ta\,:x\in H,\,t\in\R,\,\iprod{x+t\frac{a}{\norm a}}a\ge1,\,\iprod{x+t\frac{a}{\norm a}}{a+\eps\theta}\le1\right\} =\\
 & \left\{ x+ta\,:x\in H,\,0\le t\le\frac{-\eps\iprod x{\theta}\norm a}{\iprod a{a+\eps\theta}}\right\} =\\
 & \left\{ x+ta\,:x\in H,\,\iprod x{\theta}\le0,\,0\le t\le\frac{-\eps\iprod x{\theta}\norm a}{\iprod a{a+\eps\theta}}\right\} .
\end{align*}
Thus, 
\[
K_{-}\left(\eps\right)=\left\{ x+ta\,:x\in H,\,\iprod x{\theta}\le0,\,0\le t\le\frac{-\eps\iprod x{\theta}\norm a}{\iprod a{a+\eps\theta}}\right\} \cap{\rm int}\left(K\right).
\]

Let 
\begin{align*}
O_{-}\left(\eps\right): & =\left\{ x\in H\,:\,\iprod x{\theta}\le0,\,\left[x,\,x+\frac{-\eps\iprod x{\theta}\norm a}{\iprod a{a+\eps\theta}}a\right]\cap{\rm int}\left(K\right)\neq\emptyset\right\} .
\end{align*}
For $x\in H$ such that $\iprod x{\theta}\le0$, we have that 
\[
\frac{-\eps\iprod x{\theta}\norm a}{\iprod a{a+\eps\theta}}=\frac{\eps\left|\iprod x{\theta}\right|\norm a}{\iprod a{a+\eps\theta}}=\frac{\left|\iprod x{\theta}\right|\norm a}{\iprod aa\eps^{-1}+\iprod a{\theta}}
\]
decrease to $0$ as $\eps\searrow0$. Thus,  $O\left(\eps\right)$
shrinks to 
\begin{align*}
O_{-}\left(0\right) & =\left\{ x\in H\,:\,\iprod x{\theta}\le0,\,\left[x,\,x\right]\cap{\rm int}\left(K\right)\neq\emptyset\right\} \\
 & =\left\{ x\in H\cap{\rm int}\left(K\right)\,:\,\iprod x{\theta}\le0\right\} .
\end{align*}

For $x\in O_{-}\left(\eps\right)$, let $0\le t_{1}\left(\eps,\,x\right)\le t_{2}\left(\eps,\,x\right)\le\frac{-\eps\iprod x{\theta}}{\iprod a{a+\eps\theta}}\norm a$
be defined such that 
\[
\left\{ x+ta\,:\,0\le t\le\frac{-\eps\iprod x{\theta}\norm a}{\iprod a{a+\eps\theta}}\right\} \cap{\rm int}\left(K\right)=\left\{ x+ta\,:\,t_{1}\left(\eps,\,x\right)<t<t_{2}\left(\eps,\,x\right)\right\} .
\]
Then, by Fubini's theorem, we have 
\begin{align*}
\int_{K_{-}\left(\eps\right)}\iprod y{e_{j}}\d y= & \int_{O_{-}\left(\eps\right)}\int_{t_{1}\left(\eps,\,x\right)}^{t_{2}\left(\eps,\,x\right)}\iprod{x+t\frac{a}{\norm a}}{e_{j}}\d t\d x\\
= & \int_{O_{-}\left(\eps\right)}\int_{t_{1}\left(\eps,\,x\right)}^{t_{2}\left(\eps,\,x\right)}\iprod x{e_{j}}\d t\d x+\int_{O_{-}\left(\eps\right)}\int_{t_{1}\left(\eps,\,x\right)}^{t_{2}\left(\eps,\,x\right)}\iprod{t\frac{a}{\norm a}}{e_{j}}\d t\d x.
\end{align*}
We analyze each of the above terms, separately, as follows.

Firstly, we have that 
\begin{align*}
\left|\int_{O_{-}\left(\eps\right)}\int_{t_{1}\left(\eps,\,x\right)}^{t_{2}\left(\eps,\,x\right)}\iprod{t\frac{a}{\norm a}}{e_{j}}\d t\d x\right|\le & \int_{O_{-}\left(\eps\right)}\int_{t_{1}\left(\eps,\,x\right)}^{t_{2}\left(\eps,\,x\right)}t\d t\d x\\
\le & \int_{O_{-}\left(\eps\right)}\int_{0}^{\frac{-\eps\iprod x{\theta}\norm a}{\iprod a{a+\eps\theta}}}t\d t\d x\\
\le & \frac{1}{2}\frac{\eps^{2}\norm a^{2}}{\iprod a{a+\eps\theta}^{2}}\int_{O_{-}\left(\eps\right)}\iprod x{\theta}^{2}\d x.
\end{align*}
Since $O_{-}\left(\eps\right)$ is bounded and shrinks as $\eps$
decreases, we conclude that
\[
\lim_{\eps\searrow0}\frac{1}{\eps}\int_{O_{-}\left(\eps\right)}\int_{t_{1}\left(\eps,\,x\right)}^{t_{2}\left(\eps,\,x\right)}\iprod{t\frac{a}{\norm a}}{e_{j}}\d t\d x=0.
\]
Secondly, we have that 
\[
\frac{\int_{O_{-}\left(\eps\right)}\int_{t_{1}\left(\eps,\,x\right)}^{t_{2}\left(\eps,\,x\right)}\iprod x{e_{j}}\d t\d x}{\eps}=\int_{H}\frac{\left(t_{2}\left(x,\,\eps\right)-t_{1}\left(x,\,\eps\right)\right)\iprod x{e_{j}}{\bf 1}_{O_{-}\left(\eps\right)}\left(x\right)}{\eps}\d x.
\]
Fix $\eps_{0}>0$. For $\eps_{0}>\eps>0$, we have that 
\[
\left|\frac{\left(t_{2}\left(x,\,\eps\right)-t_{1}\left(x,\,\eps\right)\right)\iprod x{e_{j}}{\bf 1}_{O_{-}\left(\eps\right)}\left(x\right)}{\eps}\right|\le\frac{\left|\iprod x{\theta}\right|\norm a}{\iprod aa-\eps_{0}\left|\iprod a{\theta}\right|}\left|\iprod x{e_{j}}\right|{\bf 1}_{O_{-}\left(\eps_{0}\right)},
\]
where the function on the right hand side is integrable. 

Suppose $x\notin O_{-}\left(0\right)$. Then, $\frac{\left(t_{2}\left(x,\,\eps\right)-t_{1}\left(x,\,\eps\right)\right)\iprod x{e_{j}}{\bf 1}_{O_{-}\left(\eps\right)}\left(x\right)}{\eps}\rightarrow0$
as $\eps\searrow0$ since ${\bf 1}_{O_{-}\left(\eps\right)}\left(x\right)=0$
for small $\eps>0$. For $x\in O_{-}\left(0\right)$, we have $t_{1}\left(x\right)=0$
and $t_{2}\left(x\right)=\frac{-\eps\iprod x{\theta}\norm a}{\iprod a{a+\eps\theta}}$
for sufficiently small $\eps$. We conclude that, as $\eps\searrow0$,
\[
\frac{\left(t_{2}\left(x,\,\eps\right)-t_{1}\left(x,\,\eps\right)\right)\iprod x{e_{j}}{\bf 1}_{O_{-}\left(\eps\right)}\left(x\right)}{\eps}\rightarrow\frac{-\iprod x{\theta}\iprod x{e_{j}}}{\norm a}{\bf 1}_{O_{-}\left(0\right)}\left(x\right).
\]
By Lebesgue's dominated convergence theorem, we have 
\begin{align*}
 & \lim_{\eps\searrow0}-\frac{\int_{K_{-}\left(\eps\right)}\iprod x{e_{j}}\d x}{\eps}\\
= & \lim_{\eps\searrow0}-\frac{\int_{O_{-}\left(\eps\right)}\int_{t_{1}\left(\eps,\,x\right)}^{t_{2}\left(\eps,\,x\right)}\iprod x{e_{j}}\d t\d x}{\eps}\\
= & \frac{1}{\norm a}\int_{K\cap H\cap\left\{ \iprod x{\theta}\le0\right\} }\iprod x{\theta}\iprod x{e_{j}}\d x.
\end{align*}
Via the same argument, one also shows that 
\[
\lim_{\eps\searrow0}\frac{\int_{K_{+}\left(\eps\right)}\iprod x{e_{j}}\d x}{\eps}=\frac{1}{\norm a}\int_{K\cap H\cap\left\{ \iprod x{\theta}\ge0\right\} }\iprod x{\theta}\iprod x{e_{j}}\d x.
\]
Thus we conclude that 

\[
\lim_{\eps\searrow0}\frac{U_{j}\left(a+\eps\theta\right)-U_{j}\left(a\right)}{\eps}=\frac{1}{\norm a}\int_{K\cap H}\iprod x{\theta}\iprod x{e_{j}}\d x.
\]
This completes the proof of \eqref{eq:DU}. 

Next, we show that $DU\left(a\right)$ and $\nabla\delta\left(a\right)$
are $C^{k}$ functions.

Pick $v\in{\rm int}\left(K\right)\cap H$. Let $\sigma_{a}$ be the
normalized Haar measure on $S\left(a\right)=\text{\ensuremath{\Sph}}^{n-1}\cap a^{\perp}$.
Then 
\begin{align}
\int_{K\cap H}x\d x & =\left(n-1\right)\left|B_{2}^{n-1}\right|\int_{S\left(a\right)}\int_{0}^{\rho_{v}\left(\theta\right)}r^{n-2}\left(v+r\theta\right)\d r\d{\sigma_{a}\left(\theta\right)}\nonumber \\
 & =\left|B_{2}^{n-1}\right|\int_{S\left(a\right)}\left(\rho_{v}^{n-1}\left(\theta\right)v+\frac{n-1}{n}\rho_{v}^{n}\left(\theta\right)\theta\right)\d{\sigma_{a}\left(\theta\right).}\label{eq:GradientDelta}
\end{align}

Fix $a_{0}\in\R^{n}$ so that ${\rm int}\text{\ensuremath{\left(K\right)}}\cap\text{\ensuremath{\left\{ \iprod x{a_{0}}=1\right\} } \ensuremath{\neq}\ensuremath{\emptyset}}$
and let $v_{0}\in{\rm int}\left(K\right)\cap\left\{ \iprod x{a_{0}}=1\right\} $.
By Fact \ref{fact:smooth}, $\left(v,\,\theta\right)\rightarrow\rho_{v}\left(\theta\right)$
is $C^{k}$, and hence the function $F_{a_{0}}:\R^{n}\times O_{n}\to\R^{n}$
defined by 
\[
\left(v,\,T\right)\mapsto\left|B_{2}^{n-1}\right|\int_{S\left(a_{0}\right)}\left(\rho_{v}^{n-1}\left(T\theta\right)v+\frac{n-1}{n}\rho_{v}^{n}\left(T\theta\right)T\theta\right)\d{\sigma_{a_{0}}\left(\theta\right)}
\]
is also $C^{k}$. We can find a smooth function $a\mapsto\left(v\left(a\right),\,T\left(a\right)\right)$
in a neighborhood of $a_{0}$ so that $v\left(a\right)\in{\rm int}\left(K\right)\cap\left\{ \iprod xa=1\right\} $
and $T\left(a\right)S\left(a_{0}\right)=\Sph^{n-1}\cap a^{\perp}$.
Indeed, for $a$ close to $a_{0}$, we define the unique two-dimensional
rotation $T\left(a\right)$ satisfying $T\left(a\right)\frac{a_{0}}{\norm{a_{0}}}=\frac{a}{\norm a}$
and $T\left(a\right)v=v$ for all $v\in{\rm span}\left(a,\,a_{0}\right)^{\perp}$.
In particular, $a\mapsto T\left(a\right)$ is a smooth function around
$a_{0}$. Also, $T\left(a\right)\left(S\left(a_{0}\right)\right)=S\left(a\right)$.
Let $v\left(a\right)$ be the projection of $v_{0}$ onto $\left\{ \iprod xa=1\right\} $.
In other words, 
\[
v\left(a\right):=v_{0}-\iprod{v_{0}}{\frac{a}{\norm a}}\frac{a}{\norm a}+\frac{a}{\norm a^{2}},
\]
which is again smooth when $a\neq0$. Also, $v\left(a_{0}\right)=v_{0}$,
and $v\left(a\right)\in{\rm int}\left(K\right)$ if $a$ is close
to $a_{0}$.

Next, we express $\nabla\delta$ in terms of $v\left(a\right)$ and
$T\left(a\right)$: By \eqref{eq:GradientDelta} we have 
\begin{align*}
\nabla\delta\left(a\right) & =\int_{K\cap\left\{ \iprod xa=1\right\} }x\d x\\
 & =\frac{1}{\norm a}\left|B_{2}^{n-1}\right|\int_{S\left(a\right)}\left(\rho_{v\left(a\right)}^{n-1}\left(\theta\right)v\left(a\right)+\frac{n-1}{n}\rho_{v\left(a\right)}^{n}\left(\theta\right)\theta\right)\d{\sigma_{a}\left(\theta\right)}\\
 & =\frac{1}{\norm a}\left|B_{2}^{n-1}\right|\int_{S\left(a_{0}\right)}\left(\rho_{v\left(a\right)}^{n-1}\left(T\left(a\right)\theta\right)v\left(a\right)+\frac{n-1}{n}\rho_{v\left(a\right)}^{n}\left(T\left(a\right)\theta\right)T\left(a\right)\theta\right)\d{\sigma_{a_{0}}\left(\theta\right)}\\
 & =\frac{1}{\norm a}F_{a_{0}}\left(v\left(a\right),\,T\left(a\right)\right).
\end{align*}
We conclude that $\nabla\delta\left(a\right)$ is $C^{k}$ and thus
$\delta\left(a\right)$ is $C^{k+1}$.

Recall that $\delta\text{\ensuremath{\left(\theta,\,d\right)}}=\left|K\cap\left\{ \iprod x{\theta}\ge d\right\} \right|$.
Consider the function from $\Sph^{n-1}\times\R$ to $\Sph^{n-1}\times\R$
defined by 
\[
\left(\theta,\,d\right)\mapsto\left(\theta,\,\delta\left(\frac{1}{d}\theta\right)\right)=\left(\theta,\,\delta\left(\theta,\,d\right)\right).
\]
By the above, it is $C^{k+1}$ whenever ${\rm int}\left(K\right)\cap\left\{ \iprod x{\theta}=d\right\} \neq\emptyset$.
Thus, its inverse function $\left(\theta,\,d\left(\theta,\,\delta\right)\right)$
is also $C^{k+1}$ for $\left(\theta,\,\delta\right)\in\Sph^{n-1}\times\left[0,\,\left|K\right|\right]$.
Repeating the same argument as for $\nabla\delta\left(a\right)$ implies
that $U\left(a\right)$ is also $C^{k+1}$.

Recall that if $d\left(\theta,\,\delta\right)>0$, 
\[
x_{K}\left(\theta,\,\delta\right)=\frac{1}{\delta}\int_{K\cap\left\{ \iprod x{\theta}\ge d\left(\theta,\,\delta\right)\right\} }x\d x=\frac{1}{\delta}U\left(\frac{\theta}{d\left(\theta,\,\delta\right)}\right).
\]
Therefore, for a fixed $0<\delta<\left|K\right|$, and $\theta$ such
that $d\left(\theta,\,\delta\right)>0$, the function $\theta\mapsto\frac{x_{K}\left(\theta,\,\delta\right)}{\norm{x_{K}\left(\theta,\,\delta\right)}}$
is $C^{k+1}$. Moreover, it is invertible since $\mfloat{\delta}K$
is strictly convex. Thus its inverse, denoted by $G_{\delta}:\Sph^{n-1}\to\Sph^{n-1}$
is also $C^{k+1}$. Therefore, the radial function of $\mfloat{\delta}K$,
which is given by $\rho\left(\theta\right)=\norm{x\left(G_{\delta}\left(\theta\right),\,\delta\right)}$
is also $C^{k+1}$.

Finally, we need to show that $\theta\rightarrow x_{K}\left(\theta,\,\delta\right)$
is $C^{k+1}$ whenever $d\left(\theta,\,\delta\right)\le0$. Indeed,
we may choose some vector $v\in\R^{n}$ and consider $M_{\delta}\left(v+K\right)$.
Then, $x_{K}\left(\theta,\,\delta\right)=x_{v+K}\left(\theta,\,\delta\right)-v$.
Clearly, we can always choose $v$ such that, for $v+K$, $d\left(\theta,\,\delta\right)>0$.
Thus, following the same argument, we can show $x_{v+K}\left(\theta,\,\delta\right)$
is $C^{k+1}$. As a consequence, $x_{K}\left(\theta,\,\delta\right)$
is $C^{k+1}$. Therefore, we conclude that $\rho\left(\theta\right)$
is $C^{k+1}$ on $\Sph^{n-1}$. By Fact \eqref{fact:smooth}, the
boundary of $\mfloat{\delta}K$ is $C^{k+1}$.
\end{proof}

\section{\label{sec:ASA}Relation to p-affine surface area}

This section is devoted to the proof of Theorem \ref{thm:main}. 

\subsection{Preliminary results }

For the proof of Theorem \ref{thm:main}, we will need a few preliminary
results.\\

First, we focus on $\text{\ensuremath{\mfloat{\delta}{\rho B_{2}^{n},\phi}}}$,
where $\rho B_{2}^{n}$ is the Euclidean ball centered at $0$ and
with radius, and $\phi\left(x\right)$ is a constant function. By
symmetry, we know that $\text{\ensuremath{\mfloat{\delta}{\rho B_{2}^{n},\phi}}}$
is again a Euclidean ball with the same center. Let $\Delta\left(\rho,\,\delta\right)$
be the difference of the radius of $\rho B_{2}^{n}$ and $\text{\ensuremath{\mfloat{\delta}{\rho B_{2}^{n},\phi}}}$.
If $\phi:\rho B_{2}^{n}\rightarrow\left(0,\infty\right)$, is a constant
function, $\phi(x)=s$, for all $x\in\rho B_{2}^{n}$, then, we define
$\Delta\left(\rho,\,\delta,\,s\right)$ to be the difference of radius
of $\rho B_{2}^{n}$ and $\text{\ensuremath{\mfloat{\delta}{\rho B_{2}^{n},s}}}$.
One easily verifies that 
\begin{align}
\Delta\left(\rho,\,\delta,\,s\right) & =\Delta\left(\rho,\,\frac{\delta}{s}\right).\label{eq:DeltaEquivalence}
\end{align}

\begin{prop}
\label{prop:floatingball}$\lim_{\delta\searrow0}\Delta\left(\rho,\,\delta\right)/\delta^{\frac{2}{n+1}}\rho^{\frac{n+1}{n-1}}=c_{n}$,
where $c_{n}=\frac{1}{2}\frac{n+1}{n+3}\left(\frac{n+1}{\left|B_{2}^{n-1}\right|}\right)^{\frac{2}{n+1}}$. 
\end{prop}

\begin{proof}
We denote $h\left(\rho,\,\delta\right)$ to be height of the cap of
$\rho B_{2}^{n}$ which has volume $\delta$. To be specific, $h\left(\rho,\,\delta\right)$
satisfies the equality 
\[
\delta=\left|B_{2}^{n-1}\right|\int_{0}^{h\left(\rho,\delta\right)}g^{n-1}\left(t\right)\d t,
\]
where $g\left(t\right)=\left(\rho^{2}-\left(\rho-t\right)^{2}\right)^{1/2}$.
Moreover, 
\begin{align*}
g\left(t\right) & =\left(\rho^{2}-\left(\rho-t\right)^{2}\right)^{1/2}=\rho\left(1-\left(1-t/\rho\right)^{2}\right)^{1/2}=\rho\left(2-t/\rho\right)^{1/2}\left(t/\rho\right)^{1/2}.
\end{align*}
We have 
\begin{align*}
\delta & =\left|B_{2}^{n-1}\right|\rho^{n-1}\int_{0}^{h\left(\rho,\,\delta\right)}\left(2-t/\rho\right)^{\frac{n-1}{2}}\left(t/\rho\right)^{\frac{n-1}{2}}\d t.
\end{align*}
Thus, we have the inequality 
\begin{eqnarray*}
 &  & \hskip-15mm\left|B_{2}^{n-1}\right|\rho^{n-1}\left(2-h\left(\rho,\,\delta\right)/\rho\right)^{\frac{n-1}{2}}\int_{0}^{h\left(\rho,\,\delta\right)}\left(t/\rho\right)^{\frac{n-1}{2}}\d t\le\ \delta\\
 &  & \hskip30mm\le\ \left|B_{2}^{n-1}\right|\rho^{n-1}2^{\frac{n-1}{2}}\int_{0}^{h\left(\rho,\,\delta\right)}\left(t/\rho\right)^{\frac{n-1}{2}}\d t.
\end{eqnarray*}
Since 
\[
\int_{0}^{h\left(\rho,\,\delta\right)}\left(t/\rho\right)^{\frac{n-1}{2}}\d t=\frac{2}{n+1}h\left(\rho,\,\delta\right)^{\frac{n+1}{2}}\rho^{-\frac{n-1}{2}},
\]
we obtain 
\[
\frac{1}{2}\left(\frac{n+1}{\left|B_{2}^{n-1}\right|}\right)^{\frac{2}{n+1}}\rho^{-\frac{n-1}{n+1}}\le\frac{h\left(\rho,\,\delta\right)}{\delta^{\frac{2}{n+1}}}\le\frac{1}{2-h\left(\rho,\,\delta\right)/\rho}\left(\frac{n+1}{\left|B_{2}^{n-1}\right|}\right)^{\frac{2}{n+1}}\rho^{-\frac{n-1}{n+1}}.
\]
We conclude that 
\[
\lim_{\delta\searrow0}\frac{h\left(\rho,\,\delta\right)}{\delta^{\frac{2}{n+1}}}=\frac{1}{2}\left(\frac{n+1}{\left|B_{2}^{n-1}\right|}\right)^{\frac{2}{n+1}}\rho^{-\frac{n-1}{n+1}}.
\]
Recall that 
\[
\Delta\left(\rho,\,\delta\right)=\frac{\left|B_{2}^{n-1}\right|\int_{0}^{h\left(\rho,\,\delta\right)}tg\left(t\right)^{n-1}\d t}{\left|B_{2}^{n-1}\right|\int_{0}^{h\left(\rho,\,\delta\right)}g\left(t\right)^{n-1}\d t}.
\]
To compute the next limit, we apply twice L'Hospital's Rule, 
\begin{align*}
\lim_{h\to0}\frac{h}{\Delta} & =\lim\frac{h\int_{0}^{h}h^{n-1}\d t}{\int_{0}^{h}tg^{n-1}\d t}\overset{L}{=}\lim\frac{\int_{0}^{h}g^{n-1}\d t+hg\left(h\right)^{n-1}}{hg\left(h\right)^{n-1}}=1+\lim\frac{\int_{0}^{h}g^{n-1}\d t}{hg\left(h\right)^{n-1}}\\
 & \overset{L}{=}1+\lim\frac{\rho^{n-1}\left(2-\frac{r}{\rho}\right)^{\frac{n-1}{2}}\left(\frac{r}{\rho}\right)^{\frac{n-1}{2}}}{\rho^{n}\left(\frac{1}{\rho}\frac{n+1}{2}\left(\frac{r}{\rho}\right)^{\frac{n-1}{2}}\left(2-\frac{r}{\rho}\right)^{\frac{n-1}{2}}-\frac{1}{\rho}\frac{n-1}{2}\left(\frac{r}{\rho}\right)^{\frac{n+1}{2}}\left(2-\frac{r}{\rho}\right)^{\frac{n-3}{2}}\right)}\\
 & =1+\lim\frac{\left(2-\frac{r}{\rho}\right)}{\frac{n+1}{2}\left(2-\frac{r}{\rho}\right)-\frac{n-1}{2}\left(\frac{r}{\rho}\right)}=1+\frac{2}{n+1}=\frac{n+3}{n+1}.
\end{align*}
So, 
\[
\lim_{\delta\searrow0}\frac{\Delta\left(\rho,\,\delta\right)}{\delta^{\frac{2}{n+1}}}=\lim_{\delta\searrow0}\frac{h\left(\rho,\,\delta\right)}{\delta^{\frac{2}{n+1}}}\cdot\frac{\Delta\left(\rho,\,\delta\right)}{h\left(\rho,\,\delta\right)}=\frac{1}{2}\ \frac{n+1}{n+3}\left(\frac{n+1}{\left|B_{2}^{n-1}\right|}\right)^{\frac{2}{n+1}}\rho^{-\frac{n-1}{n+1}}.
\]
\end{proof}
We will also need the next lemma from \cite{Schuett:1990}: 
\begin{lem}
\label{difference} Let $K$ and $L$ be convex bodies in $\mathbb{R}^{n}$
such that $0\in\text{int}(L)$ and such that $L\sub K$. Then 
\[
|K|-|L|=\frac{1}{n}\int_{\partial K}\langle x,N(x)\rangle\left(1-\left|\frac{\norm{{x}_{L}}}{\norm x}\right|^{n}\right)\d{\mu_{K}\left(x\right)},
\]
where $x_{L}$ is the unique point in the intersection $\partial L\cap\left[0,x\right]$. 
\end{lem}

For the next lemma we need a notion that was introduced in \cite{Schuett:1990}.
Let $K$ be a convex body in $\R^{n}$ and let $x\in\partial K$ be
such that $N_{K}(x)$ is unique. We put $r(x)$ to be the radius of
the biggest Euclidean ball contained in $K$ that touches $K$ in
$x$, 
\[
r(x)=\max\{\rho:B_{2}^{n}(x-\rho N_{K}(x),\rho)\sub K\}.
\]
If $N_{K}(x)$ is not unique, $r(x)=0$. It was shown in \cite[Lemma 5]{Schuett:1990}
that for any convex body $K$ in $\mathbb{R}^{n}$ and any $0\leq\alpha<1$,
\begin{equation}
\int_{\partial K}r(x)^{-\alpha}d\mu(x)<\infty.\label{finite}
\end{equation}

\begin{lem}
\label{Lebesgue} Let $K$ be a convex body in $\mathbb{R}^{n}$.
Let $x\in\partial K$ and let $x_{M,\delta}=\partial\left(\mwfloat{\delta}K{\phi}\right)\cap\left[0,x\right]$.
Then 
\[
\frac{\iprod x{N_{K}\left(x\right)}}{\delta^{\frac{2}{n+1}}}\ \left(1-\left|\frac{\norm{x_{M,\delta}}}{\norm x}\right|^{n}\right)\leq c\ n\ r(x)^{-\frac{n-1}{n+1}},
\]
where $c$ is a constant independent of $x$ and $\delta$. 
\end{lem}

\begin{proof}
Let $x_{F,\delta}=\partial\left(F_{\delta}\left(K,\phi\right)\right)\cap\left[0,x\right]$.
By Theorem \ref{thm:floating_metro_iso}, we have that $F_{\delta}\left(K,\phi\right)\sub M_{\delta}\left(K,\phi\right)$
and hence $\|x_{\F,\delta}\|\leq\|x_{M,\delta}\|$. Therefore 
\[
\frac{\iprod x{N_{K}\left(x\right)}}{\delta^{\frac{2}{n+1}}}\ \left(1-\left|\frac{\norm{x_{M,\delta}}}{\norm x}\right|^{n}\right)\leq\frac{\iprod x{N_{K}\left(x\right)}}{\delta^{\frac{2}{n+1}}}\ \left(1-\left|\frac{\norm{x_{F,\delta}}}{\norm x}\right|^{n}\right)
\]
and it was shown in \cite{Schuett:1990}, Lemma 8, that the latter
is smaller than or equal $c\ n\ r(x)^{-\frac{n-1}{n+1}}$. 
\end{proof}
\noindent The next lemma was proved in \cite{Schuett:1990}. There,
and in the proof of the main theorem, we need the indicatrix of Dupin
(see, e.g., \cite{Schuett:2003}). A theorem of Alexandrov \cite{Alexandroff}
and Busemann and Feller \cite{Buseman1935} shows that the indicatrix
of Dupin exists almost everywhere on $\partial K$ and is an ellipsoid
or an elliptic cylinder. We also use the notation $C(r,h)$ for the
cap of a Euclidean ball with radius $r$ and height $h$. 
\begin{lem}
\cite{Schuett:1990}\label{distance} Let $K$ be a convex body in
$\mathbb{R}^{n}$ with $0\in\partial K$ and $N_{K}(0)=-e_{n}=(0,\cdots,0,-1)$.
Suppose the indicatrix of Dupin at $0$ exists and is an $(n-1)$-dimensional
sphere with radius $\sqrt{\rho}$. Let $\xi$ be an interior point
of $K$. \vskip 2mm (i) Let $H$ be the hyperplane orthogonal to
$N_{K}(0)$ and passing through $z$ in $[0,\xi]$. We put $z_{n}=\langle z,e_{n}\rangle$.
Then we have for $0\leq z_{n}\leq\rho$, 
\[
\left|K\cap H^{+}\right|\leq f(z_{n})^{n-1}\left|C(\rho,z_{n})\right|.
\]

\noindent (ii) Let $d=\text{dist}\left(z,B_{2}^{n}(\rho\ e_{n},\rho)^{C}\right)$.
There is $\varepsilon_{0}>0$ such that we have for all $z\in[0,\xi]$
with $\|z\|\leq\varepsilon_{0}$ 
\[
d\leq z_{n}\leq d+\frac{2\ d^{2}}{\rho\langle\frac{\xi}{\|\xi\|},N_{K}(0)\rangle^{2}}.
\]

\noindent (iii) There is $\varepsilon_{0}>0$ and an absolute constant
$c>0$ such that for all $z\in[0,\xi]$ with $\|z\|\leq\varepsilon_{0}$
and all hyperplanes $H$ passing through $z$ 
\[
\left|K\cap H^{+}\right|\geq f(\gamma)^{-n+1}\left|C(\rho,d(1-c(f(\gamma)-1))\right|.
\]

\noindent Here, $\gamma=2\sqrt{2\ \rho\ d}$ and $f$ is a monotone
function on $\mathbb{R}^{+}$ such that $\lim_{t\rightarrow}f(t)=1$. 
\end{lem}

\begin{lem}
\label{lem: UnaffectedCut} Let $K\sub\R^{n}$ be a convex body. Moreover,
we assume that $0\in\partial K$ and that $N_{K}(0)=-e_{n}$ is the
unique outer normal to $\partial K$ at $0$. Let $\phi:K\rightarrow\left(0,\infty\right)$
be a continuous function. We set $H_{t}^{+}=H^{+}(-e_{n},-t)=\left\{ y\,:\,\iprod y{e_{n}}<t\right\} $.
Then, for each $t>0$, there exists $r>0$ such that for any $\delta>0$,
\[
\mwfloat{\delta}K{\phi}\cap B_{2}^{n}\left(0,\,r\right)=\mwfloat{\delta}{K\cap H_{t}^{+}}{\phi}\cap B_{2}^{n}\left(0,\,r\right).
\]
\end{lem}

\begin{proof}
It is obvious that 
\[
\mwfloat{\delta}{K\cap H_{t}^{+}}{\phi}\cap B_{2}^{n}\left(0,\,r\right)\sub\mwfloat{\delta}K{\phi}\cap B_{2}^{n}\left(0,\,r\right).
\]
Therefore, it is sufficient to show the other inclusion. Let $d\ge0$.
Observe that if $\left(\theta,\,d\right)$ is sufficiently close to
$\left(-e_{n},\,0\right)$, then $H^{+}(\theta,-d)\cap K\sub H_{t}^{+}$,
where $H^{+}(\theta,-d)=\left\{ y\,:\langle y,-\theta\rangle<d\right\} $.
As noted in (\ref{stetig}), the function $d\left(\theta,\delta\right)$
is continuous in $\left(\theta,\,\delta\right)$. Therefore, there
exists $\delta_{0}>0$ and $\eps>0$ such that 
\begin{equation}
K\cap H^{+}\left(\theta,\,d\left(\theta,\,\delta\right)\right)\sub H_{t}^{+},\label{eq:UnaffectedCutCond1}
\end{equation}
for $\norm{\theta-\left(-e_{n}\right)}<\eps$ and $0\le\delta<\delta_{0}$.
For each $x$ in the interior of $K$, let $\delta\left(x\right)$
be the value such that $x\in\partial\mwfloat{\delta\left(x\right)}K{\phi}$
and $\theta\left(x\right)$ denote the unique outer normal at $x$
of $\mwfloat{\delta\left(x\right)}K{\phi}$. 

\noindent ${\bf Claim:}$ For any $\delta_{0}>0$ and $\eps>0$, there
exists $r>0$ such that $\delta\left(x\right)<\delta_{0}$ and $\norm{\theta\left(x\right)-\left(-e_{n}\right)}<\eps$,
for $x\in{\rm int}\left(K\right)\cap B_{2}^{n}\left(0,\,r\right)$.

\noindent Indeed, note that $\mwfloat{\delta_{0}}K{\phi}$ is strictly
contained in $K$. Thus, $0\notin\mwfloat{\delta_{0}}K{\phi}$. Since
$\mwfloat{\delta_{0}}K{\phi}$ is convex, there exists $r>0$ so that
$B_{2}^{n}\left(0,\,r\right)\cap\mwfloat{\delta_{0}}K{\phi}=\emptyset$.
Then, $\delta\left(x\right)<\delta_{0}$ for $x\in{\rm int}\left(K\right)\cap B_{2}^{n}\left(0,\,r\right)$.

\noindent It remains to show that there exists $r>0$ such that $\norm{\theta\left(x\right)-\left(-e_{n}\right)}<\eps$
for ${\rm int}\left(K\right)\cap B_{2}^{n}\left(0,\,r\right)$. Suppose
it is false. Then there exists a sequence $(x_{k})_{k\in\mathbb{N}}$
in ${\rm int}\left(K\right)$ such that $x_{k}\rightarrow0$ and such
that $\norm{\theta\left(x_{k}\right)-\left(-e_{n}\right)}>\eps$.
By the compactness of $\Sph^{n-1}$, we may replace $(x_{k})_{k\in\mathbb{N}}$
by a subsequence, again denoted by $(x_{k})_{k\in\mathbb{N}}$, so
that $\theta\left(x_{k}\right)$ converges to some $\theta_{1}\neq-e_{n}$.
Moreover, $\delta\left(x_{k}\right)\rightarrow0$ since the first
claim is true. Continuity of $h_{\mwfloat{\delta}K{\phi}}\left(\theta\right)$
implies that $h_{\mwfloat{\delta\left(x_{k}\right)}K{\phi}}\left(\theta\left(x_{k}\right)\right)\rightarrow h_{K}\left(\theta_{1}\right)$.
As $-e_{n}$ is the unique outer normal to $\partial K$ in $0$,
$h_{K}\left(\theta_{1}\right)>\iprod 0{\theta_{1}}=0$. Therefore,
we obtain a contradiction, as $h_{\mwfloat{\delta\left(x_{k}\right)}K{\phi}}\left(\theta\left(x_{k}\right)\right)=\iprod{x_{k}}{\theta\left(x_{k}\right)}$,
which converges to $0$ as $x_{k}\rightarrow0$. This completes the
proof of the claim.

\noindent Hence, with the assumptions on $\delta_{0}$ and $\eps$,
we conclude that there exists $r>0$ such that for $x\in{\rm int}\left(K\right)\cap B_{2}^{n}\left(0,\,r\right)$,
\[
K\cap H^{+}\left(\theta\left(x\right),\,d\left(\theta\left(x\right),\,\delta\left(x\right)\right)\right)\sub H_{t}^{+}.
\]

\noindent Let $x\in M_{\delta}\left(K,\,\phi\right)\cap B_{2}^{n}\left(0,\,r\right)$.
Since $x\in{\rm int}\left(K\right)\cap B_{2}^{n}\left(0,\,r\right)$,
\begin{align*}
K\cap H^{+}\left(\theta(x),\,d\left(\theta(x),\,\delta(x)\right)\right) & \sub H_{t}^{+},
\end{align*}
and thus $x\in\mwfloat{\delta\left(x\right)}{K\cap H_{t}^{+}}{\phi}$.
Moreover, notice that $\delta\left(x\right)\ge\delta$ and hence we
have 
\[
\mwfloat{\delta\left(x\right)}{K\cap H_{t}^{+}}{\phi}\sub\mwfloat{\delta}{K\cap H_{t}^{+}}{\phi}.
\]
Hence, $x\in\mwfloat{\delta}{K\cap H_{t}^{+}}{\phi}$. Therefore,
$\mwfloat{\delta}K{\phi}\cap B\left(0,\,r\right)\sub\mwfloat{\delta}{K\cap H_{t}^{+}}{\phi}\cap B\left(0,\,r\right).$
\end{proof}

\subsection{Proof of Theorem \ref{thm:main}}

\noindent Recall that $x_{M}$ is the unique point in $\partial\left(\mwfloat{\delta}K{\phi}\right)\cap\left[0,x\right]$.
We will sometimes write in short $x_{M}$ for $x_{M,\delta}$. By
Lemmas \ref{difference} and \ref{Lebesgue}, we have that 
\[
\lim_{\delta\rightarrow0}\frac{\left|K\right|-\left|\mwfloat{\delta}K{\phi}\right|}{\delta^{\frac{2}{n+1}}}=\frac{1}{n}\int_{\partial K}\lim_{\delta\rightarrow0}\delta^{-\frac{2}{n+1}}\iprod x{N_{K}\left(x\right)}\left(1-\left|\frac{\norm{x_{M}}}{\norm x}\right|^{n}\right)\d{\mu_{K}\left(x\right)}.
\]
For $x\in\partial K$ fixed, the goal is to understand 
\[
\lim_{\delta\searrow0}\frac{1}{n}\int_{\partial K}\delta^{-\frac{2}{n+1}}\iprod x{N_{K}\left(x\right)}\left(1-\left|\frac{\norm{x_{M}}}{\norm x}\right|^{n}\right)\d{\mu_{K}\left(x\right)}.
\]
As $x$ and $x_{M}$ are collinear and as for all $0\leq a\leq1$,
\[
1-na\le\left(1-a\right)^{n}\le1-na+\frac{n\left(n-1\right)}{2}a^{2},
\]
we get for $\delta$ sufficiently small that 
\begin{eqnarray}
 &  & \hskip-10mm\frac{\|x-x_{M}\||}{\|x\|}\left(1-\frac{n-1}{2}\frac{\|x-x_{M}\||}{\|x\|}\right)\leq\frac{1}{n}\left(1-\left|\frac{\norm{x_{M}}}{\norm x}\right|^{n}\right)=\nonumber \\
 &  & \hskip20mm\frac{1}{n}\left[1-\left(1-\frac{\norm{x-x_{M}}}{\norm x}\right)^{n}\right]\leq\frac{\|x-x_{M}\||}{\|x\|}.
\end{eqnarray}
\textbf{(i)} We assume first that the indicatrix of Dupin at $x\in\partial K$
is an ellipsoid. In fact, by a change of the coordinate system, we
may also assume that $x=0$ and $N_{K}\left(0\right)=-e_{n}$. Let
$\zeta\in\R^{n}$ be the origin in the previous coordinate system.
Let $y_{M,\delta}:=\partial\left(\mwfloat{\delta}K{\phi}\right)\cap\left[0,\zeta\right]$.
Notice that $\norm{y_{M,\delta}}=\norm{x-x_{M,\delta}}$ and that
$y_{M,\delta}\rightarrow0$ as $\delta\searrow0$. Thus 
\begin{equation}
\lim_{\delta\searrow0}\iprod x{N_{K}\left(x\right)}\frac{\norm{x-x_{M,\delta}}}{\norm x}=\lim_{\delta\searrow0}\iprod{\zeta}{e_{n}}\frac{\norm{y_{M,\delta}}}{\norm{\zeta}}=\lim_{\delta\searrow0}\iprod{y_{M,\,\delta}}{e_{n}}.\label{limes}
\end{equation}
There exists a volume preserving positive definite linear transform
$T$ such that $N_{TK}\left(0\right)=-e_{n}$ and such that the indicatrix
of Dupin at $0$ becomes a Euclidean ball with radius $\sqrt{\rho}$
(see, e.g., equation $\left(5\right)$ in \cite{Schuett:2003}). Moreover,
$\rho$ satisfies 
\[
\kappa_{K}\left(0\right)=\frac{1}{\rho^{n-1}}.
\]
Let $H^{+}$ be the half space such that 
\[
\delta=\int_{K\cap H^{+}}\phi\left(y\right)\d y\quad\ \ \text{and}\ \ \ y_{M,\,\delta}=\frac{\int_{K\cap H^{+}}y\phi\left(y\right)\d y}{\delta}.
\]
As $T$ is volume preserving, $\int_{TK\cap TH^{+}}\phi\left(T^{-1}y\right)\d y=\delta$,
and thus 
\begin{eqnarray*}
Ty_{M,\,\delta} & = & \int_{K\cap H^{+}}Ty\phi\left(y\right)\d y/\delta=\int_{TK\cap TH^{+}}y\phi\left(T^{-1}y\right)\d y/\delta\\
 & \in & \partial\mwfloat{\delta}{TK}{\phi\circ T^{-1}}.
\end{eqnarray*}
As a consequence we have 
\[
\left[0,\,T\zeta\right]\cap\partial\mwfloat{\delta}{TK}{\phi\circ T^{-1}}=Ty_{M,\,\delta},
\]
\[
\phi\left(T^{-1}0\right)=\phi\left(0\right),
\]
and 
\[
\iprod{Ty_{M,\delta}}{e_{n}}=\iprod{y_{M,\delta}}{Te_{n}}=\iprod{y_{M,\delta}}{e_{n}}.
\]
Hence we have reduced the problem to the case when the indicatrix
of Dupin at $0\in\partial K$ is a Euclidean sphere with radius $\sqrt{\rho}$
and $\kappa_{K}\left(0\right)=\frac{1}{\rho^{n-1}}$.

\noindent Moreover, $\partial K$ can be approximated in $0$ by a
Euclidean ball $B_{2}^{n}(\rho e_{n},\rho)$ of radius $\rho$ and
center $\rho e_{n}$ in the following sense (see, e.g., \cite[Proof of Lemma 23]{Schuett:2004}):

\noindent Let $\varepsilon>0$ be given. Let $B_{2}^{n}\left((1-\varepsilon)\rho e_{n},(1-\eps)\rho\right)$
be the Euclidean ball centered at $(1-\varepsilon)\rho e_{n}$ whose
radius is $\left(1-\eps\right)\rho$. Similarly, let $B_{2}^{n}\left((1+\varepsilon)\rho e_{n},(1+\eps)\rho\right)$
be the Euclidean ball centered at $(1+\varepsilon)\rho$ with radius
$(1+\varepsilon)\rho$. Then, 
\begin{eqnarray*}
\hskip-10mm0\in\partial\left[B_{2}^{n}(\rho e_{n},\rho)\right],\ \ 0\in\partial\left[B_{2}^{n}\left((1-\varepsilon)\rho e_{n},(1-\eps)\rho\right)\right],\\
\hskip20mm0\in\partial\left[B_{2}^{n}\left((1+\varepsilon)\rho e_{n},(1+\eps)\rho\right)\right],
\end{eqnarray*}
and 
\[
N_{B_{2}^{n}(\rho e_{n},\rho)}=N_{B_{2}^{n}\left((1-\varepsilon)\rho e_{n},(1-\eps)\rho\right)}=N_{B_{2}^{n}\left((1+\varepsilon)\rho e_{n},(1+\eps)\rho\right)}=-e_{n}
\]
and (see, e.g., \cite[Proof of Lemma 23]{Schuett:2004}) there exists
$\Delta_{\varepsilon}^{0}$ such that for $0<t<\Delta_{\eps}^{0}$,
the half-space $H_{t}^{+}=\{y:\langle y,e_{n}\rangle\leq t\}$ determined
by the hyperplane orthogonal to $e_{n}$ through the point $te_{n}$
is such that 
\begin{eqnarray}
 &  & \hskip-20mmH_{t}^{+}\cap B_{2}^{n}\left((1-\varepsilon)\rho e_{n},(1-\eps)\rho\right)\subseteq H_{t}^{+}\ \cap\ K\nonumber \\
 &  & \hskip20mm\subseteq H_{t}^{+}\ \cap B_{2}^{n}\left((1+\varepsilon)\rho e_{n},(1+\eps)\rho\right).\label{eq:SameAsCap}
\end{eqnarray}
By continuity of $\phi$ there exists $s>0$ such that for all $y\in\text{int}(B_{2}^{n}(0,s))$,
\begin{equation}
\left(1-\eps\right)\phi(0)\leq\phi(y)\leq\left(1+\eps\right)\phi(0).\label{stetig-1}
\end{equation}

\noindent We will apply Lemma \ref{lem: UnaffectedCut} with $t=\Delta_{\eps}^{0}$
simultaneously to $K$, $B_{2}^{n}\left((1-\varepsilon)\rho e_{n},(1-\eps)\rho\right)$
and $B_{2}^{n}\left((1+\varepsilon)\rho e_{n},(1+\eps)\rho\right)$
with weights $\phi,\,\left(1-\eps\right)\phi\left(0\right),$ and
$\left(1+\eps\right)\phi\left(0\right)$ respectively.

\noindent Let $H_{\Delta_{\eps}}^{+}=\{y:\langle y,e_{n}\rangle\leq\Delta_{\varepsilon}\}$.
We choose $\Delta_{\varepsilon}\leq\Delta_{\varepsilon}^{0}$ so small
that 
\[
H_{\Delta_{\eps}}^{+}\ \cap\ B_{2}^{n}\left((1+\varepsilon)\rho e_{n},(1+\eps)\rho\right)\ \sub B_{2}^{n}(0,\min\left\{ s,\,r\right\} ),
\]
where $r$ is given by Lemma \ref{lem: UnaffectedCut}. We denote
\begin{align*}
d_{M,\,\delta}^{-}=\text{dist}\left(y_{M,\,\delta},B_{2}^{n}\left((1-\varepsilon)\rho e_{n},(1-\eps)\rho\right)^{c}\right)
\end{align*}
and 
\begin{align*}
\quad d_{M,\,\delta}^{+}=\text{dist}\left(y_{M,\,\delta},B_{2}^{n}\left((1+\varepsilon)\rho e_{n},(1+\eps)\rho\right)^{c}\right).
\end{align*}
Boundedness of $\phi$ on $B_{2}^{n}\left(0,s\right)$ and (\ref{eq:SameAsCap})
imply that for $\delta\ge0$, 
\begin{eqnarray*}
 &  & \hskip-10mm\mwfloat{\delta}{B_{2}^{n}\left((1-\varepsilon)\rho e_{n},(1-\eps)\rho\right)\cap H_{\Delta_{\eps}}^{+}}{\left(1-\eps\right)\phi\left(0\right)}\sub\mwfloat{\delta}{K\cap H_{\Delta_{\eps}}^{+}}{\phi}\\
 &  & \hskip20mm\sub\mwfloat{\delta}{B_{2}^{n}\left((1+\varepsilon)\rho e_{n},(1+\eps)\rho\right)\cap H_{\Delta_{\eps}}^{+}}{\left(1+\eps\right)\phi\left(0\right)}.
\end{eqnarray*}
By Lemma \ref{lem: UnaffectedCut} and the choice of $\Delta_{\eps}$
we have 
\begin{eqnarray*}
 &  & \hskip-5mm\mwfloat{\delta}{B_{2}^{n}\left((1-\varepsilon)\rho e_{n},(1-\eps)\rho\right)}{\left(1-\eps\right)\phi\left(0\right)}\cap H_{\Delta_{\eps}}^{+}\ \sub\ \mwfloat{\delta}K{\phi}\cap H_{\Delta_{\eps}}^{+}\\
 &  & \hskip10mm\sub\ \mwfloat{\delta}{B_{2}^{n}\left((1+\varepsilon)\rho e_{n},(1-\eps)\rho\right)}{\left(1+\eps\right)\phi\left(0\right)}\cap H_{\Delta_{\eps}}^{+}.
\end{eqnarray*}
Choose $\delta$ so small that $y_{M,\delta}\in H_{\Delta_{\eps}}^{+}$.
Then 
\[
y_{M,\,\delta}\ \notin{\rm int}\ \left(\mwfloat{\delta}{B_{2}^{n}\left((1-\varepsilon)\rho e_{n},(1-\eps)\rho\right)}{\left(1-\eps\right)\phi\left(0\right)}\right)
\]
and 
\[
y_{M,\,\delta}\ \in\ {\rm int}\ \left(\mwfloat{\delta}{B_{2}^{n}\left((1-\varepsilon)\rho e_{n},(1+\eps)\rho\right)}{\left(1+\eps\right)\phi\left(0\right)}\right).
\]
Thus, we conclude that 
\[
d_{M,\,\delta}^{-}\le\Delta\left(\left(1-\eps\right)\rho,\,\left(1-\eps\right)\delta\phi\left(0\right)\right)\ \ {\rm and}\ \ d_{M,\,\delta}^{+}\ge\Delta\left(\left(1+\eps\right)\rho,\,\left(1+\eps\right)\delta\phi\left(0\right)\right),
\]
where $\Delta\left(\left(1+\eps\right)\rho,\,\left(1+\eps\right)\delta\phi\left(0\right)\right)$
and $\Delta\left(\left(1-\eps\right)\rho,\,\left(1-\eps\right)\delta\phi\left(0\right)\right)$
are the differences of the radii of $\left(1+\eps\right)\rho B_{2}^{n}$
and $\text{\ensuremath{\mfloat{\delta}{\rho B_{2}^{n},(1+\varepsilon)\phi(0)}}}$,
and of $\left(1-\eps\right)\rho B_{2}^{n}$ and $\text{\ensuremath{\mfloat{\delta}{\rho B_{2}^{n},(1-\varepsilon)\phi(0)}}}$,
respectively. Applying Lemma \ref{distance}(ii) with $z=y_{M,\,\delta}$
and Proposition \ref{prop:floatingball} for sufficiently small $\delta$,
yields 
\[
\left(1-\eps\right)^{\frac{n+1}{n-1}+\frac{2}{n+1}}\le\frac{\iprod{y_{M,\,\delta}}{e_{n}}}{c_{n}\delta^{\frac{2}{n+1}}\rho^{-\frac{n-1}{n+1}}\phi(0)^{\frac{2}{n+1}}}\le\left(1+\eps\right)^{\frac{n+1}{n-1}+\frac{2}{n+1}}.
\]
Since $\eps>0$ can be chosen arbitrary, we obtain, also using \eqref{limes},
\[
\lim_{\delta\rightarrow0}\phi(x)^{\frac{2}{n+1}}\ \iprod x{N_{K}\left(x\right)}\frac{\norm{x-x_{M,\delta}}}{\norm x\delta^{\frac{2}{n+1}}}=c_{n}\ \rho(x)^{-\frac{n-1}{n+1}}=c_{n}\ \kappa_{K}(x)^{\frac{1}{n+1}}.
\]

\vskip 2mm (ii) Now we assume that $x$ is such that the indicatrix
of Dupin at $x$ is an elliptic cylinder. We will show that then 
\[
\lim_{\delta\rightarrow0}\iprod x{N_{K}\left(x\right)}\frac{\norm{x-x_{M,\delta}}}{\norm x\delta^{\frac{2}{n+1}}}=0.
\]
We only need to show that $\lim_{\delta\rightarrow0}\iprod x{N_{K}\left(x\right)}\frac{\norm{x-x_{M,\delta}}}{\norm x\delta^{\frac{2}{n+1}}}\leq0$.

\noindent We may assume that the first $k$ axes of the elliptic cylinder
have infinite lengths and the others not. Then, as above (see, e.g.,
\cite[Proof of Lemma 23]{Schuett:2004}) for all $\varepsilon>0$
there is an approximating ellipsoid $\mathcal{E}$ and $\Delta_{\varepsilon}$
such that the hyperplane $H\left(N_{K}(x),x-\Delta_{\varepsilon})N_{K}(x)\right)$
orthogonal to $N_{K}(x)$ through the point $x-\Delta_{\varepsilon}N_{K}(x)$
is such that 
\begin{eqnarray*}
H^{+}\left(N_{K}(x),x-\Delta_{\varepsilon})N_{K}(x)\right)\ \cap\ \mathcal{E}\subseteq H^{+}\left(N_{K}(x),x-\Delta_{\varepsilon})N_{K}(x)\right)\ \cap\ K
\end{eqnarray*}
and such that the lengths of the $k$ first principal axes of $\mathcal{E}$
are larger than $\frac{1}{\varepsilon}$. As noted above, there is
a support hyperplane $H_{\delta}$ to ${F}_{\delta}\left(K,\phi\right)$
such that $x_{F,\delta}\in H_{\delta}$ and such that $\delta=\int_{K\cap H_{\delta}^{+}}\phi(y)dy$
\cite{Werner2002}. Then 
\[
\delta\geq\min_{y\in K}\phi(y)|K\cap H_{\delta}^{+}|\geq\min_{y\in K}\phi(y)|\mathcal{E}\cap H_{\delta}^{+}|.
\]
As above, we may assume that the approximating ellipsoid $\mathcal{E}$
is a Euclidean ball with radius $\rho=\rho(x)$ where $\rho\geq\frac{1}{\varepsilon}$.
Then 
\begin{eqnarray*}
\iprod x{N_{K}\left(x\right)}\frac{\norm{x-x_{M,\delta}}}{\norm x\delta^{\frac{2}{n+1}}} & \leq & \iprod x{N_{K}\left(x\right)}\frac{\norm{x-x_{F,\delta}}}{\norm x\delta^{\frac{2}{n+1}}}\\
 & \leq & \frac{\langle\frac{x}{\|x\|},N_{K}(x)\rangle\ \norm{x-x_{F,\delta}}}{\left(\min_{y\in K}\phi(y)\right)^{\frac{2}{n+1}}\left(|B_{2}^{n}(x-\rho N_{K}(x),\rho)\cap H_{\delta}^{+}|\right)^{\frac{2}{n+1}}}\\
 & \leq & \frac{\rho^{-\frac{n-1}{n+1}}}{c_{n}\left(\min_{y\in K}\phi(y)\right)^{\frac{2}{n+1}}}.
\end{eqnarray*}
The last inequality can be shown using similar methods as in the case
(i). Or, one notices that we are precisely in the situation of Lemmas
7 and 10 of \cite{Schuett:1990} where exactly this estimate is proved.
As $\rho$ is arbitrarily small, the proof is completed.

\bibliographystyle{amsplain_NoDash}
\bibliography{refs}

\providecommand{\bysame}{\leavevmode\hbox to3em{\hrulefill}\thinspace}
\providecommand{\MR}{\relax\ifhmode\unskip\space\fi MR }
\providecommand{\MRhref}[2]{%
  \href{http://www.ams.org/mathscinet-getitem?mr=#1}{#2}
}
\providecommand{\href}[2]{#2}
\begin{thebibliography}{10}

\bibitem{Alexandroff}
A.~D. Alexandroff, \emph{Almost everywhere existence of the second differential
  of a convex function and some properties of convex surfaces connected with
  it}, Leningrad State Univ. Annals [Uchenye Zapiski] Math. Ser. \textbf{6}
  (1939), 3--35.

\bibitem{Andrews:1996}
B. Andrews, \emph{Contraction of convex hypersurfaces by their affine normal},
  J. Differential Geom. \textbf{43} (1996), no.~2, 207--230.

\bibitem{Andrews:1999}
B. Andrews, \emph{The affine curve-lengthening flow}, J. Reine Angew. Math.
  \textbf{506} (1999), 43--83.

\bibitem{Artstein-Avidan:2012}
S. Artstein-Avidan, B. Klartag, C. Sch{\"u}tt, and E. Werner, \emph{Functional
  affine-isoperimetry and an inverse logarithmic {S}obolev inequality}, J.
  Funct. Anal. \textbf{262} (2012), no.~9, 4181--4204.

\bibitem{Auerbach1938}
H. Auerbach, \emph{Sur un probl\`eme de {M}. {U}lam concernant l'\'equilibre
  des corps flottants}, Studia Mathematica \textbf{7} (1938), no.~1, 121--142.

\bibitem{Barany_Larman:1988}
I. B\'ar\'any and D.~G. Larman, \emph{Convex bodies, economic cap coverings,
  random polytopes}, Mathematika \textbf{35} (1988), no.~2, 274--291.

\bibitem{Besau2016}
F. Besau, M. Ludwig, and E.~M. Werner, \emph{{Weighted floating bodies and
  polytopal approximation}}, To appear in Transactions of the AMS.

\bibitem{Besau_Werner:2016}
F. Besau and E.~M. Werner, \emph{The spherical convex floating body}, Adv.
  Math. \textbf{301} (2016), 867--901.

\bibitem{Besau_Werner:2018}
F. Besau and E.~M. Werner, \emph{The floating body in real space forms}, J.
  Diff. Geom., in press (2018).

\bibitem{Blaschke:1923}
W. Blaschke, \emph{Vorlesungen {\"{u}}ber differentialgeometrie ii, affine
  differentialgeometrie}, Springer-Verlag, Berlin, 1923.

\bibitem{Boeroeczky:2000a}
K. B{\"o}r{\"o}czky, Jr., \emph{Approximation of general smooth convex bodies},
  Adv. Math. \textbf{153} (2000), no.~2, 325--341.

\bibitem{Boeroeczky:2000}
K. B{\"o}r{\"o}czky, Jr., \emph{Polytopal approximation bounding the number of
  {$k$}-faces}, J. Approx. Theory \textbf{102} (2000), no.~2, 263--285.

\bibitem{Buseman1935}
H. Busemann and W. Feller, \emph{Kr\"ummungseigenschaften {K}onvexer
  {F}l\"achen}, Acta Math. \textbf{66} (1936), no.~1, 1--47.

\bibitem{Caglar:2014}
U. Caglar and E.~M. Werner, \emph{Divergence for {$s$}-concave and log concave
  functions}, Adv. Math. \textbf{257} (2014), 219--247.

\bibitem{CFG91}
H.~T. Croft, K.~J. Falconer, and R.~K. Guy, \emph{Unsolved problems in
  geometry}, Problem Books in Mathematics, Springer-Verlag, New York, 1991,
  Unsolved Problems in Intuitive Mathematics, II.

\bibitem{Dupin:1822}
C. Dupin, \emph{Application de g{\'{e}}om{\'{e}}trie et de m{\'{e}}chanique},
  Paris, 1822.

\bibitem{Fleury:2007}
B. Fleury, O. Gu{\'e}don, and G. Paouris, \emph{A stability result for mean
  width of {$L_p$}-centroid bodies}, Adv. Math. \textbf{214} (2007), no.~2,
  865--877.

\bibitem{Gardner:2006}
R.~J. Gardner, \emph{Geometric tomography}, second ed., Encyclopedia of
  Mathematics and its Applications, vol.~58, Cambridge University Press,
  Cambridge, 2006.

\bibitem{Gilbert91}
E.~N. Gilbert, \emph{How things float}, The American Mathematical Monthly
  \textbf{98} (1991), no.~3, 201--216.

\bibitem{Grote_Werner:2018}
J. Grote and E.~M. Werner, \emph{Approximation of smooth convex bodies by
  random polytopes}, To appear in Electronic Journal of Probability (2018),
  \href{https://arxiv.org/abs/1706.07623}{arXiv:1709.02429}.

\bibitem{Gruber:1988}
P.~M. Gruber, \emph{Volume approximation of convex bodies by inscribed
  polytopes}, Math. Ann. \textbf{281} (1988), no.~2, 229--245.

\bibitem{Gruber:1993}
P.~M. Gruber, \emph{Aspects of approximation of convex bodies}, Handbook of
  convex geometry, {V}ol.\ {A}, {B}, North-Holland, Amsterdam, 1993,
  pp.~319--345.

\bibitem{Gruber:2007}
P.~M. Gruber, \emph{Convex and discrete geometry}, Grundlehren der
  Mathematischen Wissenschaften [Fundamental Principles of Mathematical
  Sciences], vol. 336, Springer, Berlin, 2007.

\bibitem{Huang2017}
H. Huang and B.~A. Slomka, \emph{{Approximations of convex bodies by
  measure-generated sets}}, To appear in Geom. Dedicata (2017),
  \href{http://arxiv.org/abs/1706.07112}{arXiv:1706.07112}.

\bibitem{Hug:1996}
D. Hug, \emph{Contributions to affine surface area}, Manuscripta Math.
  \textbf{91} (1996), no.~3, 283--301.

\bibitem{Ivaki:2014}
M.~N. Ivaki, \emph{On the stability of the {$p$}-affine isoperimetric
  inequality}, J. Geom. Anal. \textbf{24} (2014), no.~4, 1898--1911.

\bibitem{Ivaki:2013a}
M.~N. Ivaki and A. Stancu, \emph{Volume preserving centro-affine normal flows},
  Comm. Anal. Geom. \textbf{21} (2013), no.~3, 671--685.

\bibitem{Leichtweiss:1986}
K. Leichtwei{\ss}, \emph{Zur {A}ffinoberfl\"ache konvexer {K}\"orper},
  Manuscripta Math. \textbf{56} (1986), no.~4, 429--464.

\bibitem{Leichtweiss89}
K. Leichtwei\ss, \emph{Bemerkungen zur {D}efinition einer erweiterten
  {A}ffinoberfl\"ache von {E}. {L}utwak}, Manuscripta Math. \textbf{65} (1989),
  no.~2, 181--197.

\bibitem{Leichtweiss:1998}
K. Leichtwei\ss, \emph{Affine geometry of convex bodies}, Johann Ambrosius
  Barth Verlag, Heidelberg, 1998. \MR{1630116}

\bibitem{Lovasz:2007}
L. Lov\'asz and S. Vempala, \emph{The geometry of logconcave functions and
  sampling algorithms}, Random Structures Algorithms \textbf{30} (2007), no.~3,
  307--358.

\bibitem{Ludwig:1999a}
M. Ludwig, \emph{Asymptotic approximation of smooth convex bodies by general
  polytopes}, Mathematika \textbf{46} (1999), no.~1, 103--125. \MR{1750407}

\bibitem{Ludwig:2010}
M. Ludwig and M. Reitzner, \emph{A classification of {${\rm SL}(n)$} invariant
  valuations}, Ann. of Math. (2) \textbf{172} (2010), no.~2, 1219--1267.

\bibitem{Lutwak:1991}
E. Lutwak, \emph{Extended affine surface area}, Adv. Math. \textbf{85} (1991),
  no.~1, 39--68.

\bibitem{Lutwak:1996}
E. Lutwak, \emph{The {B}runn-{M}inkowski-{F}irey theory. {II}. {A}ffine and
  geominimal surface areas}, Adv. Math. \textbf{118} (1996), no.~2, 244--294.

\bibitem{Lutwak:2002a}
E. Lutwak, D. Yang, and G. Zhang, \emph{The {C}ramer-{R}ao inequality for star
  bodies}, Duke Math. J. \textbf{112} (2002), no.~1, 59--81.

\bibitem{Lutwak:2004b}
E. Lutwak, D. Yang, and G. Zhang, \emph{Moment-entropy inequalities}, Ann.
  Probab. \textbf{32} (2004), no.~1B, 757--774.

\bibitem{Lutwak:2005a}
E. Lutwak, D. Yang, and G. Zhang, \emph{Cram\'er-{R}ao and moment-entropy
  inequalities for {R}enyi entropy and generalized {F}isher information}, IEEE
  Trans. Inform. Theory \textbf{51} (2005), no.~2, 473--478.

\bibitem{Lutwak:1997}
E. Lutwak and G. Zhang, \emph{Blaschke-{S}antal\'o inequalities}, J.
  Differential Geom. \textbf{47} (1997), no.~1, 1--16.

\bibitem{ScotBook81}
R.~D. Mauldin (ed.), \emph{The {S}cottish {B}ook}, Birkh\"auser, Boston, Mass.,
  1981, Mathematics from the Scottish Caf\'e, Including selected papers
  presented at the Scottish Book Conference held at North Texas State
  University, Denton, Tex., May 1979.

\bibitem{Meyer1989}
M. Meyer and S. Reisner, \emph{{Characterizations of ellipsoids by
  section-centroid location}}, Geometriae Dedicata \textbf{31} (1989), no.~3,
  345--355.

\bibitem{Meyer-Reisner1991}
M. Meyer and S. Reisner, \emph{A geometric property of the boundary of
  symmetric convex bodies and convexity of flotation surfaces}, Geom. Dedicata
  \textbf{37} (1991), no.~3, 327--337.

\bibitem{Meyer:2000}
M. Meyer and E. Werner, \emph{On the {$p$}-affine surface area}, Adv. Math.
  \textbf{152} (2000), no.~2, 288--313.

\bibitem{Nomizu:1994}
K. Nomizu and T. Sasaki, \emph{Affine differential geometry}, Cambridge Tracts
  in Mathematics, vol. 111, Cambridge University Press, Cambridge, 1994,
  Geometry of affine immersions. \MR{1311248}

\bibitem{Paouris_Werner:2012}
G. Paouris and E.~M. Werner, \emph{Relative entropy of cone measures and
  {$L_p$} centroid bodies}, Proc. Lond. Math. Soc. (3) \textbf{104} (2012),
  no.~2, 253--286.

\bibitem{Reitzner:2005}
M. Reitzner, \emph{The combinatorial structure of random polytopes}, Adv. Math.
  \textbf{191} (2005), no.~1, 178--208.

\bibitem{Schneider:1993}
R. Schneider, \emph{Convex bodies: the {B}runn-{M}inkowski theory},
  Encyclopedia of Mathematics and its Applications, vol.~44, Cambridge
  University Press, Cambridge, 1993.

\bibitem{Schuett:1991}
C. Sch{\"u}tt, \emph{The convex floating body and polyhedral approximation},
  Israel J. Math. \textbf{73} (1991), no.~1, 65--77.

\bibitem{Schuett:1993}
C. Sch{\"u}tt, \emph{On the affine surface area}, Proc. Amer. Math. Soc.
  \textbf{118} (1993), no.~4, 1213--1218.

\bibitem{Schutt:1994}
C. Sch\"utt, \emph{Random polytopes and affine surface area}, Math. Nachr.
  \textbf{170} (1994), 227--249.

\bibitem{Schuett:1990}
C. Sch{\"u}tt and E. Werner, \emph{The convex floating body}, Math. Scand.
  \textbf{66} (1990), no.~2, 275--290.

\bibitem{Schuett:2003}
C. Sch{\"u}tt and E. Werner, \emph{Polytopes with vertices chosen randomly from
  the boundary of a convex body}, Geometric aspects of functional analysis,
  Lecture Notes in Math., vol. 1807, Springer, Berlin, 2003, pp.~241--422.

\bibitem{Schuett:2004}
C. Sch{\"u}tt and E. Werner, \emph{Surface bodies and {$p$}-affine surface
  area}, Adv. Math. \textbf{187} (2004), no.~1, 98--145.

\bibitem{Stancu:2002}
A. Stancu, \emph{The discrete planar {$L_0$-M}inkowski problem}, Adv. Math.
  \textbf{167} (2002), no.~1, 160--174.

\bibitem{Stancu:2003}
A. Stancu, \emph{On the number of solutions to the discrete two-dimensional
  {$L_0$-M}inkowski problem}, Adv. Math. \textbf{180} (2003), no.~1, 290--323.

\bibitem{Trudinger:2005}
N.~S. Trudinger and X.-J. Wang, \emph{The affine {P}lateau problem}, J. Amer.
  Math. Soc. \textbf{18} (2005), no.~2, 253--289.

\bibitem{Varkonyi13}
P.~L. V\'arkonyi, \emph{Neutrally floating objects of density {$\frac12$} in
  three dimensions}, Stud. Appl. Math. \textbf{130} (2013), no.~3, 295--315.

\bibitem{Werner2002}
E. Werner, \emph{The {$p$}-affine surface area and geometric interpretations},
  Rend. Circ. Mat. Palermo (2) Suppl. (2002), no.~70, part II, 367--382, IV
  International Conference in ``Stochastic Geometry, Convex Bodies, Empirical
  Measures $\&$ Applications to Engineering Science'', Vol. II (Tropea, 2001).

\bibitem{Werner:2007}
E. Werner, \emph{On {$L_p$}-affine surface areas}, Indiana Univ. Math. J.
  \textbf{56} (2007), no.~5, 2305--2323.

\bibitem{Werner:2008}
E. Werner and D. Ye, \emph{New {$L_p$} affine isoperimetric inequalities}, Adv.
  Math. \textbf{218} (2008), no.~3, 762--780.

\bibitem{Werner:2012}
E.~M. Werner, \emph{R\'enyi divergence and {$L_p$}-affine surface area for
  convex bodies}, Adv. Math. \textbf{230} (2012), no.~3, 1040--1059.

\end{thebibliography}

\end{document}